\documentclass[11pt]{article}
\usepackage{amssymb}
\usepackage{amsmath}
\usepackage{caption}
\usepackage{subcaption}
\usepackage{latexsym}
\usepackage{mathrsfs}
\usepackage[shortlabels]{enumitem}
\usepackage{amsthm}
\usepackage{graphicx}
\usepackage{caption}
\usepackage{hyperref}
\usepackage{color}
\allowdisplaybreaks[1]
\usepackage{comment}
\usepackage{float}
\usepackage{amsmath}
\usepackage{mathtools}
\DeclarePairedDelimiter{\norma}{\lVert}{\rVert}

\numberwithin{equation}{section}
\newtheorem{theorem}[equation]{Theorem}

\newtheorem{definition}[equation]{Definition}
\newtheorem{remark}[equation]{Remark}
\newtheorem{lemma}[equation]{Lemma}

\newcounter{counter_a}

\newcommand\dom{\operatorname{dom}}

\newcommand{\R}{\mathbb{R}}

\newcommand{\rd}{\mathrm{d}}

\newcommand\cD{\mathcal D}
\newcommand\cE{\mathcal E}

\newcommand\cM{\mathcal M}
\newcommand\cH{\mathcal H}

\newcommand\ov\overline

\newcommand\eps\varepsilon
\renewcommand\epsilon\varepsilon
\renewcommand\rho\varrho
\newcommand\al\alpha
\newcommand\la\lambda

\newcommand\ds\displaystyle

\newcommand\p\partial

\newcommand{\wto}{\rightharpoonup}

\title{On the eigenvalues of the biharmonic operator with Neumann boundary conditions on a thin set}
\author{Francesco Ferraresso\footnote{Cardiff University\,, School of Mathematics\,, Maths and Education Building , Senghennydd Road\,, Cardiff, CF24 4AG. Email:\,ferraressof@cardiff.ac.uk}  and Luigi Provenzano\footnote{Sapienza Universit\`a di Roma\,, Dipartimento di Scienze di Base e Applicate per l'Ingegneria\,, Via Scarpa 16\,, 00161 Roma\,, Italy. Email:\, luigi.provenzano@uniroma1.it}}

\begin{document}

\maketitle

\noindent
{\bf Abstract:}
Let $\Omega$ be a bounded domain in $\mathbb R^2$ with smooth boundary $\partial\Omega$, and let $\omega_h$ be the set of points in $\Omega$ whose distance from the boundary is smaller than $h$. We prove that the eigenvalues of the biharmonic operator on $\omega_h$ with Neumann boundary conditions converge to the eigenvalues of a limiting problem in the form of system of differential equations on $\partial\Omega$.

\vspace{11pt}


\vspace{6pt}

\section{Introduction and statement of the main result}

Let $\Omega$ be a bounded domain in $\mathbb R^2$ with smooth boundary $\partial\Omega$. For $h>0$, we define the domain $\omega_h$  as
\begin{equation}\label{omega_h}
\omega_h:=\left\{x\in\Omega: {\rm dist\ }(x,\partial\Omega)<h\right\}.
\end{equation}
We consider the Neumann eigenvalue problem for the biharmonic operator in $\omega_h$, namely
\begin{equation}\label{classic_N}
\begin{cases}
\Delta^2u_h=\mu(h)\,u_h\,,& {\rm in\ }\omega_h,\\
\partial^2_{\nu\nu}u_h=0\,, & {\rm on\ }\partial\omega_h,\\
{\rm div}_{\partial\omega_h}\!(D^2 u_h\cdot\nu)_{\partial\omega_h}+\partial_{\nu}\Delta u_h=0\,, & {\rm on\ }\partial\omega_h,
\end{cases}
\end{equation}
in the unknowns $u_h\in C^4(\omega_h)\cap C^3(\overline{\omega_h})$ (the eigenfunction) and $\mu(h)\in\mathbb R$ (the eigenvalue). Here $\nu$ denotes the outer unit normal to $\omega_h$, $D^2u_h$ denotes the Hessian of $u_h$, ${\rm div}_{\partial\omega_h}$ denotes the tangential divergence on $\partial\Omega$, and $(D^2 u_h\cdot\nu)_{\partial\omega_h}$ denotes the projection of $D^2 u_h\cdot\nu$ on the tangent space $T\partial\omega_h$.

In this paper we are interested in the asymptotic behaviour of the solutions of problem \eqref{classic_N} as $h\rightarrow 0^+$. When $h$ is close to zero, we refer to $\omega_h$ as to a {\it thin domain}, which eventually collapses to the planar curve representing $\partial\Omega$ as $h\rightarrow 0^+$, see Figure \ref{fig:1}.

The analysis of eigenvalue problems for differential operators on thin domains has attracted noticeable interest in recent years, see e.g., \cite{ferraresso_tri,arrieta_mpereira,ArrVil2,bor_freitas,brandolini,CarKhr,casado,GauGomPer,grieser,krej_3,krej_2,nakasato_mpereira,NazPerTas,mpereira_rossi} and references therein. A somehow complementary point of view is adopted in the asymptotic analysis of domains with small holes or perforations, see e.g.,  \cite{ABCM,DRM,FFT,Lan,NazGul}. Since the literature on this topic is quite vast, our list is far from being  exhaustive. In the case of linear partial differential operators of second order subject to homogeneous Neumann boundary conditions, it is well-known that it is often possible to reduce the dimension of the problem by ignoring the thin directions, see e.g., \cite{HalRau}. The rigorous mathematical justification of the corresponding asymptotic analysis ansatz is usually very delicate and relies on a set of techniques which depends on the particular problem. For the asymptotic analysis of the Neumann Laplacian on fixed disjoint domains joined by thin cylindrical tubes, or {\it dumbbell domains}, we refer to \cite{ACLC,Jim,JimKos}. The same operator has been studied on a thin neighbourhood of a graph in \cite{KucZen}, and on thin domains with oscillating boundaries in \cite{ArrVil2}.

As for higher order operators, in \cite[\S 4]{AFL17} the analysis of the biharmonic operator with Poisson coefficient $\sigma$ and Neumann boundary conditions on a thin rectangle $(0,1)\times (0,h)\subset\R^2$  has shown that the techniques used for the Laplacian can still be employed in order to reduce the dimension and find the correct limiting problem as $h\rightarrow 0^+$. However, differently from the Neumann Laplacian, in which case the limiting operator is $-d^2/dx^2$ on $(0,1)$, the eigenvalues of the Neumann biharmonic operator converge to the eigenvalues of the fourth order operator $(1 - \sigma^2) d^4/dx^4$. In fact, when $\sigma \neq 0$ the derivatives along the thin directions give a non-trivial contribution in the limit. This result casts a shadow on whether the `natural' asymptotic analysis ansatz, namely the negligible contribution of the thin directions to the limiting problem, is valid for the biharmonic operator on a thin domain.

Inspired by the previous discussion, in this article we take a further step in the analysis of the operator $\Delta^2$ with Neumann boundary conditions on general smooth bounded thin domains of $\R^2$. Note that, differently from the case of the rectangle $(0,1)\times(0,h)$, the thin domain $\omega_h$ defined by \eqref{omega_h} collapses to a closed curve.

We recall that, in applications, the biharmonic operator $\Delta^2$ is used to model the transverse vibrations of a plate of negligible thickness whose position at rest is described by the shape of the domain, according to the Kirchhoff-Love model for elasticity. The possibility of the plate to assume non-trivial displacement at the boundary is then modelled by Neumann boundary conditions, also called boundary conditions for the free plate. We refer to  \cite{bourlard_nicaise,colbois_provenzano_bi,giroire_nedlec,Nadai,nazaret} for more details on the physical justification of the problem and for historical information. See also \cite[\S 10]{weinstock_book}. 
In our analysis, the plate is thin in a second direction (the direction normal to the boundary), which eventually vanishes. Hence, in the limit, we are left with a one-dimensional vibrating curved object, which is usually referred as to a beam or a rod. Linear elasticity theory for vibrating straight rods is quite well established, see e.g., \cite{banks,banks0,weinstock_book}. For curved rods, we refer to \cite{SaPa,Ju,SaHu} for the derivation of a corresponding mathematical model. In particular, the analysis therein is carried out in the framework of linear elasticity for a three dimensional tube of small width around a curve. The model is obtained by sending the width to zero. The resulting limiting problem can be written in the form of a system which depends on the curvature of the underlying curve. In our case, we start from the Kirchhoff model for a plate (therefore a first dimensional reduction has already been performed), and then we push the remaining dimension to zero. Our results should be then compared with those of \cite{SaPa,Ju,SaHu}. We also mention \cite{gapa} where the authors consider a biharmonic eigenvalue problem on a thin multi-structure with vanishing thickness and Dirichlet boundary conditions.

Problem \eqref{classic_N} will be understood in a weak sense. Namely, we consider the following problem
\begin{equation}\label{weak_N}
\int_{\omega_h}D^2u_h:D^2\phi \, dx=\mu(h)\int_{\omega_h}u_h\,\phi \, dx\,,\ \ \ \forall \phi\in H^2(\omega_h),
\end{equation}
in the unknowns $u_h\in H^2(\omega_h)$ and $\mu(h)\in\mathbb R$. Here $D^2 u:D^2v$ denotes the standard product of Hessians $D^2 u:D^2v:=\sum_{i,j=1}^2\partial^2_{x_ix_j}\!u \,\,\partial^2_{x_ix_j}v$. Since $\Omega$ has smooth boundary, there exists $\bar h>0$ such that for all $h\in(0,\bar h)$ the domain $\omega_h$ is smooth as well. Thus, for this choice of $h$, problem \eqref{weak_N} is well-posed and admits an increasing sequence of non-negative eigenvalues diverging to $+\infty$ of the form
\[
0=\mu_1(h)=\mu_2(h)=\mu_3(h)<\mu_{4}(h)\leq\cdots\leq\mu_j(h)\leq\cdots\nearrow+\infty.
\]
The corresponding eigenfunctions can be chosen to define a Hilbert basis of $L^2(\omega_h)$. For fixed $h \in (0, \bar{h})$, due to the smoothness assumptions on $\Omega$, any solution to \eqref{weak_N} is actually a classical solution, i.e., it solves \eqref{classic_N}, see \cite[\S 2.5]{gazzola}. The eigenvalue $\mu(h)=0$ has multiplicity $3$ and the corresponding eigenspace is spanned by $\left\{1,x_1,x_2\right\}$. In other words, the eigenspace coincides with the set of polynomials of degree at most one.

For the reader's convenience, we recall the analogous problem for the Neumann Laplacian:
\begin{equation}\label{Neumann_L}
\begin{cases}
-\Delta u_h=m(h)\, u_h\,, & {\rm in\ }\omega_h,\\
\partial_{\nu}u_h=0\,, & {\rm on\ }\partial\omega_h.
\end{cases}
\end{equation}
In this case we have
\[
0=m_1(h)<m_2(h)\leq\cdots\leq m_j(h)\leq\cdots\nearrow +\infty.
\]
It is well-known that
\[
\lim_{h\rightarrow 0^+}m_j(h)=\lambda_j
\]
where $\lambda_j$ are the eigenvalues of $
-\Delta_{\partial\Omega}u=\lambda u$ on $\partial\Omega$ and $-\Delta_{\partial\Omega}$ is the Laplacian (or Laplace-Beltrami operator) on $\partial\Omega$. We refer to \cite{schatzman} for a detailed analysis of this problem in any space dimension $n\geq 2$. In the case $n=2$, the limiting problem in the arc-length parametrization of $\partial\Omega$ is just $
-u''(s)=\lambda u(s)$, $s\in(0,|\partial\Omega|)$
with $u(0)=u(|\partial\Omega|)$, $u'(0)=u'(|\partial\Omega|)$. Here $s$ is the arc-length parameter and $|\partial\Omega|$ is the length of $\partial\Omega$.

In the present article, we shall focus only on the case $n=2$. The case $n\geq 3$ can be treated essentially in the same way. However, we point out the appearance of technicalities, quite involved computations and very long formulae. We believe that the case $n=2$ already shows the main features and highlights the peculiar behaviour of the biharmonic operator under the considered singular perturbation. We shall postpone the technical details and computations for higher dimensions in a future note.

The present paper had its origin in two pivotal observations that underline the stark difference between the biharmonic operator and the Laplace operator with Neumann boundary conditions on two-dimensional thin domains. First, the result of \cite{schatzman}  cannot hold in the case of the biharmonic operator. In fact, it is well-known that the eigenvalues of the biharmonic operator $\Delta_{\partial\Omega}^2$ on $\partial\Omega$ are exactly the squares of the Laplacian eigenvalues $\lambda_j$ on $\partial\Omega$ whenever $\partial\Omega$ is sufficiently smooth, see e.g., \cite[\S 5.8]{colbois_provenzano_bi}. In particular, the first eigenvalue of $\Delta_{\partial\Omega}^2$ is $\lambda_1^2=0$, while the second is $\lambda_2^2>0$. On the other hand, when $n=2$, $\lim_{h\rightarrow 0^+}\mu_j(h)=0$ for $j=1,2,3$, hence the eigenvalues of the biharmonic operator on $\partial\Omega$ are not the limits of the eigenvalues of the biharmonic operator on $\omega_h$ as $h\rightarrow 0^+$.

A second motivation comes from explicit computations in the unit disk $\Omega=B(0,1) \subset \mathbb R^2$. In this situation, we observe that the limiting eigenvalues of problem \eqref{classic_N} are of the form $\frac{2\ell^2(\ell^2-1)^2}{1+2\ell^2}$ for $\ell\in\mathbb N$. The eigenvalue corresponding to $\ell=0$ is simple, and the associated eigenfunction constant. The eigenvalues corresponding to $\ell\geq 1$ have multiplicity two, with associated eigenfunctions lying in the linear span of $\cos(\ell s),\sin(\ell s)$, $s\in(0,2\pi)$. In particular, zero is an eigenvalue of multiplicity three, as one expects. See Subsection \ref{disk} for more details.

It is quite surprising that the index $\ell$ appears also at the denominator in the expression of the limiting eigenvalues. This suggests that the limiting problem is in the form of a system of differential equations rather than a single eigenvalue equation. This is exactly what we prove.
\begin{theorem}\label{main}
Let $\mu_j(h)$, $j\in\mathbb N\setminus\{0\}$, be the eigenvalues of problem \eqref{classic_N}. Then $\lim_{h\rightarrow 0^+}\mu_j(h)=\eta_j$ for all $j\in\mathbb N\setminus\{0\}$, where $\eta_j$ is the $j$-th eigenvalue of the following problem
\begin{equation}\label{classic_system}
\begin{cases}
u''''-2(\kappa^2 u')'-(\kappa w)''-2(\kappa w')'=\eta u, & {\rm in\ }(0,|\partial\Omega|),\\
-2w''+\kappa^2 w-\kappa u''-2(\kappa u')'=0, & {\rm in\ } (0,|\partial\Omega|),\\
u^{(k)}(0)=u^{(k)}(|\partial\Omega|), & k=0,1,2,3\\
w^{(k)}(0)=w^{(k)}(|\partial\Omega|), & k=0,1.
\end{cases}
\end{equation}
in the unknowns $u(s)$, $w(s)$ and $\eta$ (the eigenvalue). Here $s$ is the arc-length parameter describing $\partial\Omega$ and $\kappa(s)$ denotes the curvature of the boundary at the point $s\in(0,|\partial\Omega|)$.
\end{theorem}

\begin{remark}
In the case of the unit circle we have that a solution $(u,w)$ corresponding to an eigenvalue $\frac{2\ell^2(\ell^2-1)^2}{1+2\ell^2}$ is given by $u(s)=A\cos(\ell s)+B\sin(\ell s)$, $w(s)=-\frac{3\ell^2}{1+2\ell^2}u(s)$, with $A,B\in\mathbb R$.
\end{remark}

\begin{remark}
Define $\Delta_{\kappa} = \big(- \frac{\rd^2 }{\rd x^2} + \frac{\kappa^2}{2} \big)$. Note that \eqref{classic_system} can be rewritten as a single equation by setting $w = \Delta_\kappa^{-1} (\kappa u''/2 + (\kappa u')')$, thus yielding
\[
u''''-2(\kappa^2 u')'-\left(\kappa \Delta_\kappa^{-1} \left(\frac{\kappa u''}{2} + (\kappa u')'\right)\right)''-2\left(\kappa \left(\Delta_\kappa^{-1} \left(\frac{\kappa u''}{2} + (\kappa u')'\right)\right)'\right)'=\eta u.
\]
This equation is evidently different from both the free beam equation with lateral tension $\tau$, namely $u'''' - \tau u'' = \eta u$, and the buckled beam equation, namely $u'''' = \eta u''$. It seems to us that \eqref{classic_system} behaves more like a non-local free beam with lateral tension and with variable coefficients depending on the curvature $\kappa$. 
Note that the operator $\Delta_{\kappa}$ is strictly positive because $\kappa$ is not identically zero. Indeed, it is not difficult to prove that the following Poincar\'e inequality holds: $\norma{u'}^2_{L^2((0,|\partial\Omega|))} + \frac{1}{2}\norma{\kappa u}^2_{L^2((0,|\partial\Omega|))} \geq c \norma{u}^2_{L^2((0,|\partial\Omega|))}$ for all $u \in \dom(\Delta_\kappa)$ with $c>0$ independent of $u$. If instead $\kappa \equiv 0$, constant functions are in the kernel of $\Delta_\kappa$; in fact, the case $\kappa \equiv 0$ is substantially different, see Remark \ref{rmk:polygons}. 
\end{remark}

\begin{remark}
Theorem \ref{main} shows that the thin limit of the Neumann biharmonic operator is a system of equations, therefore transforming a scalar operator acting on functions of two variables into a vector operator acting on functions of one variable. In elasticity theory, transformations of this kind are not infrequent: for example, it is well-known that the limit as $t \to 0^+$ of the Reissner-Mindlin model for plates of non-negligible thickness $t > 0$  in $\R^3$ (which is in the form of a system acting on functions of three variables) is the Kirchhoff-Love model, which involves a single equation acting on functions of two variables. However, to the best of our knowledge the result in Theorem \ref{main} is the first example of a thin limit process transforming a scalar operator into a system, in a sense going in the opposite direction compared to the Reissner-Mindlin to Kirchhoff-Love singular limit.
\end{remark}

\begin{remark}
One can check that $(u_0(s),w_0(s))=(1,0)$ and $(u_i(s),w_i(s))=(x_i(s),-\nu_i(s))$, $i=1,2$, span the set of solutions of \eqref{classic_system} corresponding to $\eta=0$. Here $x_i(s)$ denotes the restriction of the coordinate function $x_i$ to $\partial\Omega$, expressed in the arc-length variable $s$, while $\nu_i(s)$ denotes the $i$-th component of the outer unit normal $\nu(s)$ at the point of $\partial\Omega$ described by $s$. This is expected from our convergence result, since the spectral projection on the zero eigenspace converge pointwise and ${\rm span} \{1, x_1, x_2\}$ is the eigenspace associated to $\mu(h) = 0$ in \eqref{weak_N} for all $h\in(0,\bar h)$.
\end{remark}

\begin{remark} \label{rmk:polygons}
The case where $\Omega$ is a polygon in $\R^2$ cannot be deduced from Theorem \ref{main} for two main reasons. First, a polygon does not have the regularity required by the tubular neighbourhood theorem (Theorem \ref{tubular0}), which we heavily exploit in the proof. Second, a polygon with straight edges has curvature $\kappa = 0$ a.e. in $\partial \Omega$. In the proof of Theorem \ref{main}, it is needed $\kappa \neq 0$ on a set of positive measure instead. Moreover, from the considerations above on the multiplicity of the zero eigenvalue of problems \eqref{classic_N} and \eqref{classic_system}, we realise that the limiting problem in the case of a polygon cannot coincide with \eqref{classic_system} with $\kappa=0$ which is exactly the closed problem for the biharmonic operator on $\partial\Omega$.
We believe that the case of the polygon is more involved and should be treated as in \cite{KucZen}, that is, by using asymptotic analysis for elliptic differential operators on fattened graphs. We plan to analyse this problem in a future note.
\end{remark}

\begin{remark} \label{rmk:Poisson} In our analysis we started from the biharmonic operator with zero Poisson ratio $\sigma$. In two dimensions, the Poisson ratio $\sigma$ is allowed to take values in $(-1,1)$. A choice of $\sigma\in(-1,1)$ will result in a change of the quadratic form in the weak formulation  \eqref{weak_N}. Namely, at the left-hand side of \eqref{weak_N} we would have
$$
\int_{\omega_h} (1-\sigma)D^2u_h:D^2\phi+\sigma\Delta u_h\Delta\phi dx.
$$
In principle, it is possible to consider this more general setting. However, even assuming that the case $\sigma=0$ has been settled, the passage to the limit in the general case is not straightforward, as one can already see in \cite{AFL17}. Therefore, for the purposes of the present paper, we shall focus only on the emblematic case $\sigma=0$, and postpone the technical analysis of $\sigma\in(-1,1)$ to a future note.
\end{remark}

The proof of Theorem  \ref{main} relies on the pointwise convergence of the resolvent operators associated with problem \eqref{weak_N} to the resolvent operator associated with problem \eqref{classic_system}. Therefore, not only we obtain pointwise convergence of the eigenvalues, but also convergence of the projections on the eigenspaces in the sense of Stummel-Vainikko (see \cite[\S 4]{AFL17}). This type of convergence is called {\it discrete convergence} in the work by Stummel \cite{stummel} and ${\mathcal P}{\mathcal Q}$-{\it convergence} in the works by Vainikko \cite{vainikko_1,vainikko_2}, see \cite{bogli} for comparison and equivalence results.

The present paper is organised as follows. In Section \ref{prel} we recall a few preliminary results on Sobolev spaces, on curvilinear coordinate systems in tubular neighbourhoods, and on standard spectral theory for problems \eqref{weak_N} and \eqref{classic_system}. Moreover, we recall the relevant results on convergence of compact operators and their spectral convergence. Section \ref{proof_main} is dedicated to the proof of our main Theorem \ref{main}. Section \ref{remarks} contains a few final remarks. In particular, it contains a brief discussion on the case of tubes of variable size and some explicit computations in the unit circle.

\section{Preliminaries and notation}\label{prel}
\subsection{Function spaces}
Let $\Omega$ be an open set in  $\mathbb R^2$.  By $H^1(\Omega)$ we denote the Sobolev space  of functions $u\in L^2(\Omega)$ with all weak  derivatives of order one in $L^2(\Omega)$.  The space $H^1(\Omega)$ is  endowed with the scalar product
\[
\langle u, v\rangle_{H^1(\Omega)}:=\int_{\Omega} \nabla u\cdot\nabla v+uv \,dx\,,\ \ \ \forall u,v\in H^1(\Omega)
\]
which induces the norm $\|u\|_{H^1(\Omega)}:=\left(\int_{\Omega} |\nabla u|^2+u^2 \, dx\right)^{\frac{1}{2}}$. \\
By $H^2(\Omega)$ we denote the Sobolev space  of functions $u\in L^2(\Omega)$ with all weak  derivatives of order one and two in $L^2(\Omega)$.   The space $H^2(\Omega)$ is  endowed with the scalar product
\[
\langle u, v\rangle_{H^2(\Omega)}:=\int_{\Omega} D^2u : D^2v+\nabla u\cdot\nabla v+uv\, dx\,,\ \ \ \forall u,v\in H^2(\Omega).
\]
which induces the norm $\|u\|_{H^2(\Omega)}:=\left(\int_{\Omega} |D^2u|^2+ |\nabla u|^2+u^2\, dx\right)^{\frac{1}{2}}$. The spaces $H^k(\Omega)$, $k\geq 2$, are naturally defined in a similar way.

When the domain is sufficiently smooth the space $H^2(\Omega)$ can be endowed with the scalar product
\[
\langle u, v\rangle_{H^2(\Omega)}:=\int_{\Omega} D^2u : D^2v+uv \, dx\,,\ \ \ \forall u,v\in H^2(\Omega).
\]
which induces the equivalent norm $\|u\|_{H^2(\Omega)}:=\left(\int_{\Omega} |D^2u|^2+u^2 \,dx\right)^{\frac{1}{2}}$. This is the case of Lipschitz domains.

The spaces $H^k(\Sigma)$ are defined in a similar way when $\Sigma$ is a Riemannian surface, with or without boundary (or, in general, a Riemannian manifold, see e.g., \cite{hebey}).

Finally, by $H^k_p((0,|\partial\Omega|))$ we denote the closure in $H^k((0,|\partial\Omega|))$ of the space $C^{\infty}_p((0,|\partial\Omega|))$, which consists of those functions in $C^{\infty}((0,|\partial\Omega|))$ with $u^{(k)}(0)=u^{(k)}(|\partial\Omega|)$ for all $k\in\mathbb N$.

\subsection{Tubular neighbourhoods of smooth boundaries and local coordinate systems}
We start this subsection by recalling the following well-known result from \cite{federer}
\begin{theorem}\label{tubular0}
Let $k\geq 2$ and let $\Omega$ be a bounded domain in $\mathbb R^2$ of class $C^k$. Then there exists $h>0$ such that every point in $\omega_h$ has a unique nearest point on $\partial\Omega$. Moreover, the function ${\rm dist}(\cdot,\partial\Omega)$ is of class $C^k$ in $\omega_h$.
\end{theorem}
Throughout the rest of the paper we shall denote by $\bar h$ the maximal possible tubular radius of $\Omega$, namely
\begin{equation}\label{barh}
\bar h:=\sup\left\{h>0:{\rm every\  point\  in\ } \omega_{h} {\rm\ has\ a\ unique\ nearest\ point\ on\ } \partial\Omega\right\}.
\end{equation}
From Theorem \ref{tubular0} it follows that if $\Omega$ is smooth, then $\bar h>0$.
\begin{figure}[h]
\centering
\includegraphics[width=0.3\textwidth]{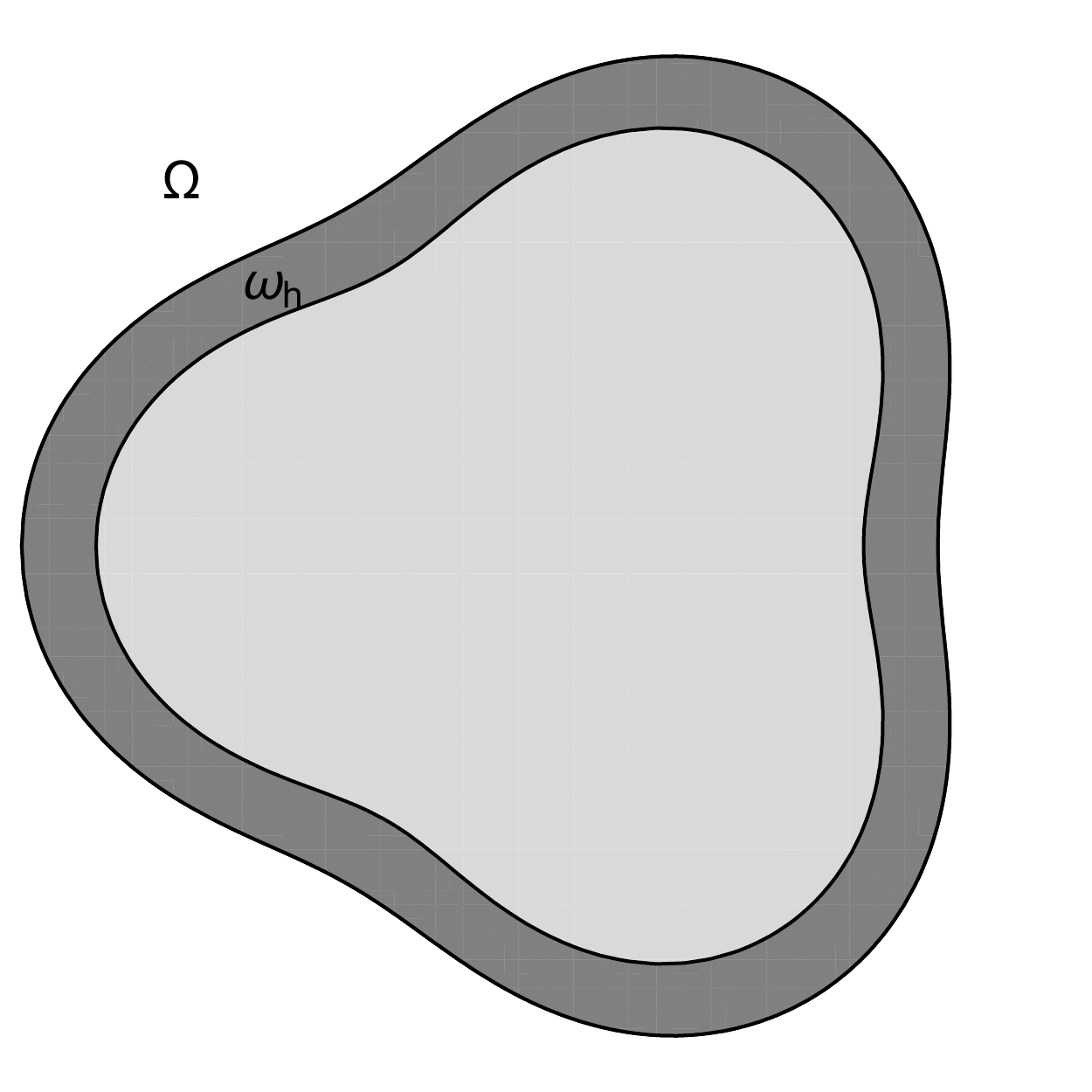}
\caption{A domain $\Omega\subset\mathbb R^2$, and $\omega_h$, a tubular neighbourhood of the boundary of size $h$.}
\label{fig:1}
\end{figure}
Let $y\in\partial\Omega$ and let $\kappa(y)$ denote the curvature of $\partial\Omega$ at $y$ with respect to the outward unit normal. In particular, if $x\in\omega_{{h}}$ and $y\in\partial\Omega$ is the nearest point to $x$ on $\partial\Omega$, then
\begin{equation}\label{bound_curv}
1-{\rm dist}(x,\partial\Omega)\kappa(y)>0,
\end{equation}
see e.g., \cite[Lemma 2.2]{lewis}.

Let $x_0\in\partial\Omega$ be fixed and let $s\in (0,|\partial\Omega|)$ be the arc-length parameter with base point $x_0$, which will correspond to $s=0$ and $s=|\partial\Omega|$. With abuse of notation, we shall often write $s$ to denote the point on $\partial\Omega$ at arc-length distance $s$ from $x_0$. In fact, we will often identify $\partial\Omega$ with the segment $(0,|\p \Omega|)$ where the endpoints have been identified.

By $\nu(s)$ we denote the outward unit normal to $\partial\Omega$ at $s$.

We introduce the map $\Phi_h$ defined by
\[
\Sigma:=\partial\Omega\times (0,1)\ni(s,t)\mapsto\Phi_h(s,t):=s-ht\nu(s)\in\omega_h.
\]
The map $\Phi_h$ is a diffeomorphism of the cylinder $\Sigma$ to $\omega_h$, see e.g., \cite[\S 2.4]{balinskybook}, see also Theorem \ref{tubular0}. In view of the identification of $\partial\Omega$ with $(0,|\p \Omega|)$, $\Phi_h$ can be thought also as a diffeomorphism of $(0,|\p \Omega|)\times(0,1)$ to $\omega_h$. In particular, $t=h^{-1}{\rm dist}(s-ht\nu(s),\partial\Omega)$.
\begin{figure}[h]
\centering
\begin{subfigure}{0.45\textwidth}
\centering
\includegraphics[width=\textwidth]{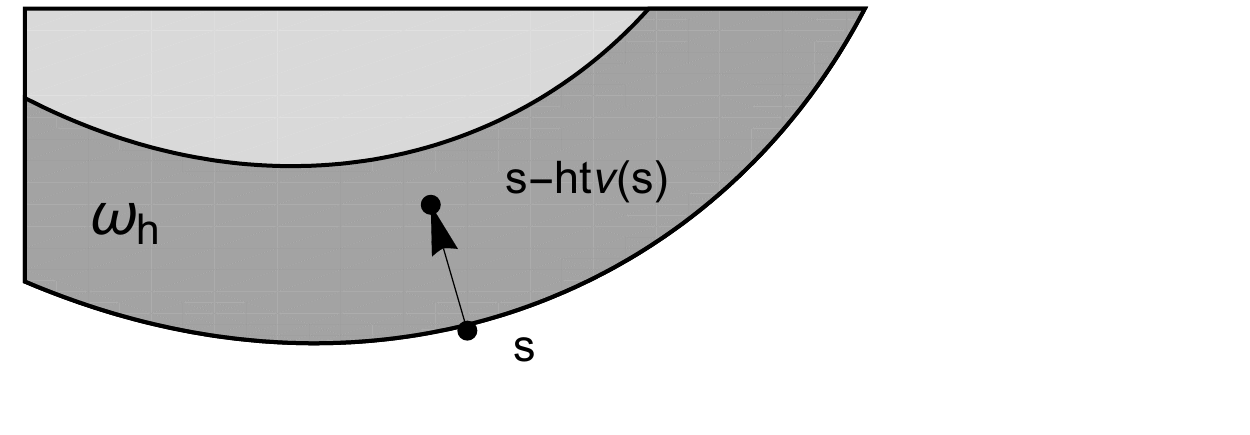}
\end{subfigure}
\begin{subfigure}{0.45\textwidth}
\centering
\includegraphics[width=\textwidth]{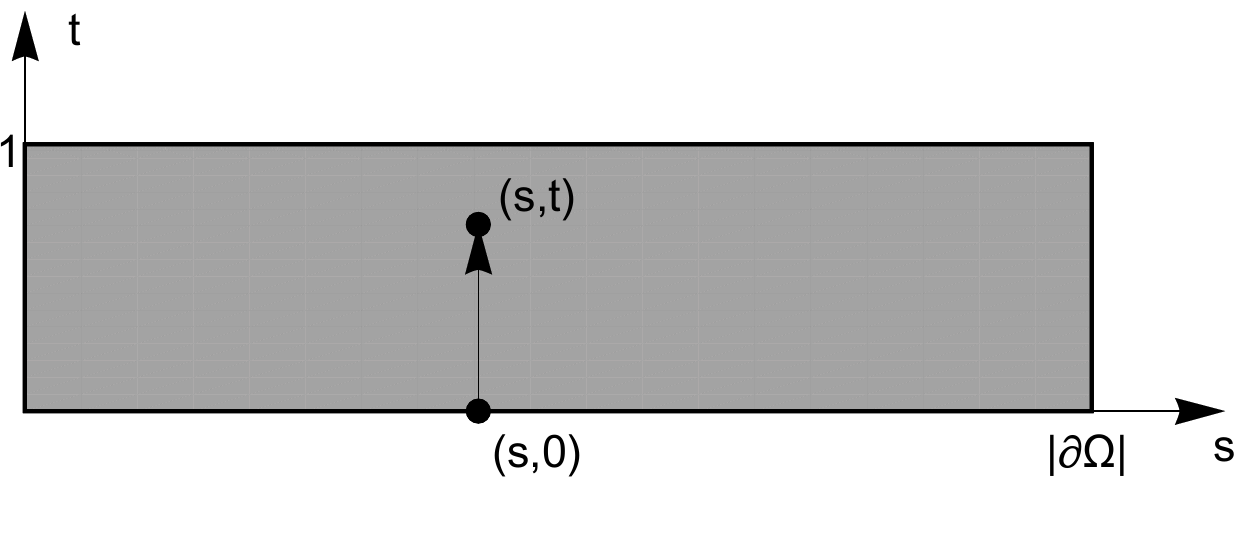}
\end{subfigure}
\end{figure}
The coordinates $(s,t)$ are sometimes called {\it curvilinear coordinates} or {\it Fermi coordinates}, which for a smooth domain are always locally defined near the boundary. In the case of $\omega_h$, they form a {\it global} coordinate system. For an integrable function $f$ on $\omega_h$, we have
\begin{equation}\label{wh_coordinates}
\int_{\omega_h}f(x)dx=\int_{\Sigma}(f\circ\Phi_h)(s,t) h(1-ht\kappa(s))dtds.
\end{equation}
Next, for smooth functions $f,g$ on $\omega_h$ we write  $\nabla f\cdot\nabla g$, $\Delta f$, and $D^2f:D^2g$ in coordinates $(s,t)$. 
Standard computations yield
\begin{equation}\label{grad_2}
\left(\nabla f\cdot\nabla g\right)\circ\Phi_h=\frac{\partial_s(f\circ\Phi_h)\cdot \partial_s(g\circ\Phi_h)}{(1-ht\kappa(s))^2}+\frac{\partial_{t}(f\circ\Phi_h)\partial_{t}(g\circ\Phi_h)}{h^2},
\end{equation}
and
\begin{multline}\label{lap_2}
(\Delta f)\circ\Phi_h
=\frac{1}{1-ht\kappa(s)}\partial_s\left(\frac{1}{1-ht\kappa(s)}\partial_s(f\circ\Phi_h)\right)\\
-\frac{(1-ht\kappa(s))}{h^2}\partial_t\left(\frac{1}{1-ht\kappa(s)}\right)\partial_t(f\circ\Phi_h)+\frac{\partial^2_{tt}(f\circ\Phi_h)}{h^2}.
\end{multline}
We refer e.g., to \cite[\S 2]{DRP} for more details.

In order to express $D^2f:D^2g$ in coordinates $(s,t)$, we shall need the following identity, which is a consequence of the so-called {\it Bochner formula}, holding for smooth functions $f,g$:
\begin{equation}\label{bochner}
D^2f:D^2g=\frac{1}{2}\left(\Delta(\nabla f\cdot\nabla g)-\nabla\Delta f\cdot\nabla g-\nabla f\cdot\nabla\Delta g\right).
\end{equation}
Thanks to \eqref{grad_2}, \eqref{lap_2} and \eqref{bochner} we obtain the following expression
\begin{multline}\label{hess_local}
(D^2f:D^2g)\circ\Phi_h\\
=\frac{\partial^2_{ss}(f\circ\Phi_h)\partial^2_{ss}(g\circ\Phi_h)}{(1-ht\kappa(s))^4}
+\frac{2\partial^2_{st}(f\circ\Phi_h)\partial^2_{st}(g\circ\Phi_h)}{h^2(1-ht\kappa(s))^2}
+\frac{1}{h^4}\partial^2_{tt}(f\circ\Phi_h)\partial^2_{tt}(g\circ\Phi_h)\\
+\frac{ht\kappa'(s)}{(1-ht\kappa(s))^5}(\partial_s(f\circ\Phi_h)\partial^2_{ss}(g\circ\Phi_h)
+\partial^2_{ss}(f\circ\Phi_h)\partial_s(g\circ\Phi_h))\\
-\frac{\kappa(s)}{h(1-ht\kappa(s))^3}(\partial^2_{ss}(f\circ\Phi_h)\partial_t(g\circ\Phi_h)
+\partial_t(f\circ\Phi_h)\partial^2_{ss}(g\circ\Phi_h))\\
+\frac{2\kappa(s)}{h(1-ht\kappa(s))^3}(\partial_s(f\circ\Phi_h)\partial^2_{st}(g\circ\Phi_h)+\partial^2_{st}(f\circ\Phi_h)\partial_s(g\circ\Phi_h))\\
-\frac{t\kappa(s)\kappa'(s)}{(1-ht\kappa(s))^4}(\partial_s(f\circ\Phi_h)\partial_t(g\circ\Phi_h)
+\partial_t(f\circ\Phi_h)\partial_s(g\circ\Phi_h))\\
+\frac{(2\kappa(s)^2(1-ht\kappa(s))^2+h^2t^2\kappa'(s)^2)}{(1-ht\kappa(s))^6}\partial_s(f\circ\Phi_h)\partial_s(g\circ\Phi_h)\\
+\frac{\kappa(s)^2}{h^2(1-t\kappa(s))^2}\partial_t(f\circ\Phi_h)\partial_t(g\circ\Phi_h).
\end{multline}
We omit the details of the computations which are standard but quite long. Now, if $f,g\in H^2(\omega_h)$, then $f\circ\Phi_h, g\circ\Phi_h\in H^2(\Sigma)$, being $\Phi_h$ smooth. By a standard density argument, identity \eqref{hess_local} holds for $f,g\in H^2(\omega_h)$.

Given a function $f\in H^2(\omega_h)$ we will use sometimes the notation $\tilde{f} := f \circ \Phi_h \in H^2(\Sigma)$ to denote the pullback of $f$ via $\Phi_h$.

\subsection{Eigenvalue problems}\label{eig_prob}
We shall recall in this subsection a few fundamental facts involving the spectral analysis of problems \eqref{weak_N} and \eqref{classic_system}.

{\bf  On problem \eqref{weak_N}.}

Let $h\in (0,\bar h)$. We shall consider a shifted version of problem \eqref{weak_N}, namely
\begin{equation}\label{weak_N_shift}
Q_h(u_h,\phi)=\zeta(h)\langle u_h,\phi\rangle_{L^2(\omega_h)},
\end{equation}
in the unknowns $u_h\in H^2(\omega_h)$ and $\zeta(h)\in\mathbb R$, where
\begin{equation}\label{quad form_N}
Q_h(u,\phi):=\int_{\omega_h}D^2u:D^2\phi+M u\phi\, dx\,,\ \ \ \forall u,\phi\in H^2(\omega_h),
\end{equation}
$\langle\cdot,\cdot\rangle_{L^2(\omega_h)}$ is the standard scalar product of $L^2(\omega_h)$, and $M>0$ is a  constant to be chosen. Note that $\zeta(h)$ is an eigenvalue of \eqref{weak_N_shift} if and only if $\mu(h)=\zeta(h)-M$ is an eigenvalue of \eqref{weak_N}, the corresponding eigenfunctions being the same.

Let $A_h$ be the unique positive self-adjoint operator associated to problem \eqref{weak_N_shift} via the equality $(A_h^{1/2} u, A_h^{1/2} \phi) = Q_h(u,\phi)$ for all $u, \phi \in \dom(A_h^{1/2}) = \dom(Q_h) = H^2(\omega_h)$. The existence of the operator $A_h$ is ensured by the second representation theorem, see \cite[Thm. VI.2.23]{kato1}.

The operator $A_h$, $h \in (0,\bar h)$, is positive, self-adjoint, with compact resolvent. Therefore $A_h$ admits an increasing sequence of positive eigenvalues
\[
0<\zeta_1(h)\leq \zeta_2(h)\leq\cdots\leq \zeta_j(h)\leq\cdots\nearrow+\infty.
\]

{\bf On problem \eqref{classic_system}}

We proceed now with the formal definition and main properties of the  operator $\tilde{A}_0$ acting in $L^2(0, |\p \Omega|)$, associated with problem \eqref{classic_system}. It will turn out that $\tilde A_0$ is related to the limit (in a suitable sense, see Subsection \ref{sp_co}) of the operators $A_h$ as $h\to 0^+$.

 Instead of considering directly $\tilde{A}_0$, it is convenient to consider, as in the case of the operators $A_h$, $h > 0$, a shifted version of $\tilde{A}_0$, namely $A_0 = \tilde{A}_0 + M$, where $M$ is the same constant of \eqref{quad form_N}. \\
Let us introduce the ordinary differential operators
\[
S_{\kappa} := \frac{\rd }{\rd s}\bigg( \kappa \frac{\rd }{\rd s} \bigg), \quad T_{\kappa} := \frac{\rd^2 }{\rd s^2}(\kappa \cdot ) \quad \Delta_{\kappa} := \bigg( -\frac{\rd^2 }{\rd s^2} + \frac{\kappa^2}{2} \bigg)
\]
with $\dom(S_{\kappa}) = \dom(T_{\kappa}) = \dom(\Delta_{\kappa}) =H^2_p((0,|\partial\Omega|))$. Then $A_0$ is the pseudodifferential operator, with $\dom(A_0)=H^4_p((0,|\p\Omega|))$, associated with the eigenvalue problem
\begin{multline} \label{eq: limit problem}
\frac{\rd^4 u }{\rd s^4 } - 2 \frac{\rd}{\rd s}\bigg(\kappa^2 \frac{\rd u}{\rd s} \bigg) - \frac{1}{2}T_{\kappa} \Delta_{\kappa}^{-1} T_{\kappa}^* u - 2 S_{\kappa} \Delta_{\kappa}^{-1} S_{\kappa} u - T_{\kappa} \Delta_{\kappa}^{-1} S_{\kappa} u - S_{\kappa} \Delta_{\kappa}^{-1} T_{\kappa}^* u  + M u\\ = \xi u,
\end{multline}
in the unknowns $u$ (the eigenfunction) and $\xi$ (the eigenvalue). Note that \eqref{eq: limit problem} is just an equivalent formulation of \eqref{classic_system}. Namely, $\xi$ is an eigenvalue of \eqref{classic_system} if and only if $\eta=\xi-M$ is an eigenvalue of \eqref{eq: limit problem}.

It will be convenient to introduce the operators
\[
\Gamma u := - \frac{1}{2} T_{\kappa} \Delta_{\kappa}^{-1} T_{\kappa}^* u - 2 S_{\kappa} \Delta_{\kappa}^{-1} S_{\kappa} u - T_{\kappa} \Delta_{\kappa}^{-1} S_{\kappa} u - S_{\kappa} \Delta_{\kappa}^{-1} T_{\kappa}^* u + M u,
\]
with $\dom(\Gamma) = \{ u \in L^2((0,|\partial\Omega|)) : \Gamma u \in L^2((0,|\partial\Omega|)), u(0) = u(|\p \Omega|), u'(0) = u'(|\p\Omega|)\} \subset H^2_p((0, |\p \Omega|))$, and
\[
L u : = \frac{\rd^4 u }{\rd s^4 } - 2 \frac{\rd}{\rd s}\bigg(\kappa^2 \frac{\rd u}{\rd s} \bigg)
\]
with $\dom(L) = H^4_p((0,|\partial\Omega|))$. Note that with these definitions
\begin{equation} \label{eq:A_0}
A_0 := L + \Gamma, \quad \quad \dom(A_0) = \dom(L) = H^4_p((0,|\partial\Omega|)).
\end{equation}

We have the following lemma.

\begin{lemma}
The operator $A_0$ defined in \eqref{eq:A_0} is self-adjoint with compact resolvent, and $A_0 > 0$ provided the constant $M$ in the definition of the operator $\Gamma$ is large enough. Therefore, $A_0$ admits an increasing sequence of positive eigenvalues $0<\xi_1\leq \xi_2\leq\cdots\leq \xi_j\leq\cdots\nearrow+\infty$.
\end{lemma}
\begin{proof}
First note that $\Gamma$ is relatively compact with respect to $L$, or equivalently, $\Gamma (L - \la)^{-1}$ is a compact operator in $L^2(0, |\p \Omega|)$ for some (and then all) $\la \in \rho(L)$. Indeed, $\Gamma$ is a pseudodifferential operator of order 2, and $(L - \la)^{-1}(L^2(0, |\p \Omega|) = \dom(L) = H^4_p((0,|\partial\Omega|))$; therefore, $\Gamma (L - \la)^{-1}$ maps $L^2(0, |\p \Omega|)$ in $H^2(0, |\p \Omega|)$. The latter compactly embeds in $L^2(0, |\p \Omega|)$ due to the Rellich-Kondrachov Theorem.\\
The $L$-compactness of $\Gamma$ implies that $\Gamma$ is relatively bounded with respect to $L$ with $L$-bound $0$; that is, there exists constant $a,b > 0$ such that (see \cite[\S IV.1.1]{kato1})
\begin{equation} \label{eq:rel_bound}
\norma{\Gamma u}_{L^2((0,|\p \Omega|))} \leq a \norma{u}_{L^2((0,|\p \Omega|))} + b \norma{L u}_{L^2((0,|\p \Omega|))}, \quad \quad u \in \dom(L)
\end{equation}
and the $L$-bound
\[
b_0 := \inf \{ b \in \R_{\geq 0} : \textup{\eqref{eq:rel_bound} holds} \}
\]
can be chosen equal to zero. Since the non-negative self-adjoint operator $L$ has compact resolvent, it follows from the stability theorem for relatively bounded perturbations (see \cite[Thm. IV.3.17]{kato1}) that $A_0$ has compact resolvent. Moreover, from \cite[Thm. V.4.11]{kato1} and the fact that $\Gamma$ is $L$-bounded with relative bound $< 1$, the operator $A_0$ is semibounded from below with lower bound $\gamma_0 \geq - \max \bigg( \frac{a(\eps)}{1 - \eps}; a(\eps) \bigg)$, where $a(\eps)$ is a choice for the constant $a$ appearing in the inequality \eqref{eq:rel_bound}, when $b = \eps$. \\
We now set, once and for all, $M : = |\gamma_0| + 1$. With this choice, $A_0 > 0$. The last claim of the Lemma is an immediate consequence of the self-adjointness of $A_0$ and of the compactness of $A_0^{-1}$.
\end{proof}

\begin{remark}\label{rem_weak_ODE}
Problem \eqref{eq: limit problem} (or, equivalently, problem \eqref{classic_system}) is understood in the weak sense as follows:
\begin{multline}\label{weak_system}
\int_0^{|\partial\Omega|}u''\phi_1''+2\kappa^2 u'\phi_1'+(\kappa w)'\phi_1'+2\kappa w'\phi_1'+2w'\phi_2'+u'(\kappa\phi_2)'+2\kappa u'\phi_2'+k^2w\phi_2 ds\\
=\eta\int_0^{|\partial\Omega|}u\phi_1 ds\,,\ \ \ \forall \phi_1\in H^2_p((0,|\partial\Omega|)), \phi_2\in H^1_p((0,|\partial\Omega|)).
\end{multline}
However, being the coefficients of the higher order terms constant, and the other coefficients smooth, any solution to \eqref{weak_system} turns out to be a classical solution.
\end{remark}

\subsection{Spectral convergence}\label{sp_co}
In this subsection we recall a few definitions of convergence of operators and their resolvents as well as related results of spectral convergence. In fact, in the next section we will prove the convergence of the eigenvalues of $A_h$ to those of $A_0$ by means of the generalised compact convergence in the sense of Stummel-Vainikko  \cite{stummel,vainikko_1,vainikko_2}.

We note that the domains $\omega_h$ vary with $h$, thus also the Hilbert spaces for $A_h$ vary as well. In order to have a common functional setting to compare the operators $A_h$ we need to introduce the notion of $\cE$-convergence of the resolvent operators.

Let $\mathcal H_{h}$, $h\in[0,\bar h)$, be a family of Hilbert spaces. We assume the existence of a family of linear operators $\cE_h\in\mathcal L(\mathcal H_0,\mathcal H_h)$ such that, for all $u_0\in\cH_0$
\begin{equation}\label{cond_ext}
\|\cE_hu_0\|_{\mathcal H_h}\rightarrow\|u_0\|_{\mathcal H_0}\,,\ \ \ {\rm as\ }h\rightarrow 0^+.
\end{equation}
\begin{definition}
Let $\mathcal H_h$ and $\cE_h$ be as above.
\begin{enumerate}[label=(\roman*)]
\item Let $u_h\in \cH_h$. We say that $u_h$ $\cE$-converges to $u_0$ if $\|u_h-\cE_h u_0\|_{\cH_h}\rightarrow 0$ as $h\rightarrow 0^+$. We write $u_h\xrightarrow{\cE}u_0$.
\item Let $B_h\in\mathcal L(\cH_h)$. We say that $B_h$ $\cE\cE$-converges to  $B_0$ if $B_h u_h\xrightarrow{\cE}B_0u_0$ whenever $u_n\xrightarrow{\cE}u_0$. We write $B_h\xrightarrow{\cE\cE}B_0$.
\item Let $B_h\in\mathcal L(\cH_h)$. We say that $B_h$ compactly converges to $B_0$, and we write $B_h\xrightarrow{C}B_0$, if the following two conditions are satisfied
\begin{enumerate}[(a)]
\item $B_h\xrightarrow{\cE\cE}B_0$ as $h\rightarrow 0^+$;
\item for any family $u_h\in \cH_h$ such that $\|u_h\|_{\cH_h}=1$ for all $h\in(0,\bar h)$, there exists a subsequence $\{B_{h_k}u_{h_k}\}_{k\in\mathbb N}$ with $h_k\rightarrow 0^+$ as $k\rightarrow+\infty$, and $u_0\in\cH_0$ such that $B_{h_k}u_{h_k}\xrightarrow{\cE}u_0$ as $k\rightarrow+\infty$.
\end{enumerate}
\end{enumerate}
\end{definition}
Compact convergence of compact operators implies spectral convergence, as stated in the following theorem.
\begin{theorem}\label{sp_conv}
Let $A_h$, $h\in[0,\bar h)$ be a family of positive, self-adjoint differential operators on $\cH_h$ with domain $\cD(A_h)\subset \cH_h$. Assume moreover that
\begin{enumerate}[label=(\roman*)]
\item The resolvent operator $B_h:=A_h^{-1}$ is compact for all $h\in[0,\bar h)$;
\item $B_h\xrightarrow{C}B_0$ as $h\rightarrow 0^+$.
\end{enumerate}
Then, if $\lambda_0$ is an eigenvalue of $A_0$, there exists a sequence of eigenvalues $\lambda_h$ of $A_h$ such that $\lambda_h\rightarrow\lambda_0$ as $h\rightarrow 0^+$. Conversely, if $\lambda_h$ is an eigenvalue
of $A_h$ for all $h\in(0,\bar h)$, and $\lambda_h\rightarrow\lambda_0$, then $\lambda_0$ is an eigenvalue of $A_0$.
\end{theorem}
We refer to \cite[Thm. 4.10]{ACLC} and \cite[Thm. 4.2]{AFL17} for the proof of Theorem \ref{sp_conv}. We also refer to \cite[Prop. 2.6]{bogli} where a spectral convergence theorem is proved for sequences of closed operators with compact resolvent. Note that an alternative approach to the spectral convergence of operators defined on variable Hilbert spaces has been proposed in the book \cite{post}. It would be interesting to implement this approach in order to recover Theorem \ref{main}.
\begin{remark}\label{rem_sp_proj}
For the purposes of the present article we have presented a simplified version of Theorem \ref{sp_conv}. Namely, we have only stated the pointwise convergence of the eigenvalues provided the resolvent operators compactly converge. Actually, if the assumptions of Theorem \ref{sp_conv} are satisfied we have a stronger spectral convergence: the projection on the generalised eigenspace $\cE$-converges pointwise. For the interested reader we refer to \cite[\S 4]{ACLC} and to \cite[\S 4]{AFL17}.
\end{remark}

\section{Proof of the main result}\label{proof_main}

The proof of Theorem \ref{main} will follow from a suitable application of Theorem \ref{sp_conv}. Through all this section, $A_h,A_0$ will be the positive, self-adjoint operators associated with problems \eqref{weak_N_shift} and \eqref{eq: limit problem}, respectively, namely the operators introduced in Subsection \ref{eig_prob}. We denote the resolvent operators of $A_h$ and $A_0$ by
\begin{equation} \label{eq: resolvents} B_0=A_0^{-1}, \quad \quad  B_h=A_h^{-1}, \quad h \in (0, \bar{h})
\end{equation}

For every $h \in (0,\bar h)$ we define $\cH_h = L^2(\omega_h;h^{-1}dx )$, $\cH_0 = L^2((0, |\p \Omega|))$, where $L^2(\omega_h;h^{-1}dx)$ is the space $L^2(\omega_h)$ endowed with the norm $\|\cdot\|_{L^2(\omega_h;h^{-1}dx)}=h^{-1/2}\|\cdot\|_{L^2(\omega_h)}$.

Let $\cE_h : \cH_0 \to \cH_h$ be the extension operator defined by $(\cE_h u\circ\Phi_h)(s,t) = u(s)$ for a.a. $s \in \partial\Omega$, $t \in (0,1)$. Note that
\[
\lim_{h\rightarrow 0^+}\norma{\cE_h u}_{\cH_h} = \norma{u}_{\cH_0},
\]
so the family of extension operators $\{\cE_h\}_{h \in (0,\bar h)}$  satisfies \eqref{cond_ext}. In particular, $\{\cE_h\}_{h \in (0,\bar h)}$ is an admissible connecting system for the family of Hilbert spaces $\{\cH_h\}_{h \in [0,\bar h)}$.

Theorem \ref{main} is a corollary of the following theorem and of Theorem \ref{sp_conv}.
\begin{theorem}\label{main_2}
Let $B_h$, $h \in [0, \bar{h})$ be defined by \eqref{eq: resolvents}. Then $B_h$ compactly converges to $B_0$ as $h \to 0^+$.
\end{theorem}
\begin{proof}
Since we are only interested in the limit as $h\rightarrow 0^+$ we may restrict to $h\in[0,\min\{1,\bar h/2\}]$. By definition of compact convergence we have to prove the following two claims:
\begin{enumerate}[label=(\roman*)]
\item for every sequence $\{f_h\}_{h \in (0,\bar h)}$, $f_h \in \cH_h$, $\cE$-convergent to $f \in \cH_0$, we have
\[
\norma{B_h f_h - \cE_h B_0 f}_{\cH_h} \to 0
\]
as $h\rightarrow 0^+$;
\item for every sequence $\{f_h\}_{h \in (0,\bar h)}$, $f_h \in \cH_h$, $\norma{f_h}_{\cH_h} = 1$, $h \in (0,\bar h)$, there exists a subsequence $\{B_{h_k}f_{h_k}\}_{k\in\mathbb N}$ with $h_k\rightarrow 0^+$ as $k\rightarrow+\infty$, and a function $u_0 \in \cH_0$ such that
    \[
    \norma{B_{h_k}f_{h_k} - \cE_{h_k}u_0}_{\cH_{h_k}} \to 0
    \]
as $k\rightarrow+\infty$.
\end{enumerate}
Consider the Poisson problem with datum $f_h\in L^2(\omega_h)$ associated with the operator $A_h$, namely
$$
\int_{\omega_h}\left( D^2u_h:D^2\varphi + M\, u_h\, \varphi \right) dx =\int_{\omega_h}f_h\,\varphi dx\,,\ \ \ \forall \varphi\in H^2(\omega_h),
$$
which is rewritten in the coordinate system $(s,t)$ (see \eqref{wh_coordinates}) as
\begin{multline} \label{eq: rescaled prob}
\int_{\Sigma} \left(D^2u_h:D^2\varphi+ M\, u_h\,\varphi\right)\circ\Phi_h(s,t) (1-ht\kappa(s))dtds \\
= \int_{\Sigma} (f_h\, \varphi)\circ\Phi_h(s,t)(1-ht\kappa(s))dtds,
\end{multline}
for all $\varphi\in H^2(\Sigma)$. Let us assume from the beginning that $f_h$ is as in the definition of compact convergence, that is, $f_h$ is uniformly bounded in the sequence of Hilbert spaces $\{\cH_h\}_{h\in(0,\bar h)}$. This means exactly that $\{\tilde{f_h}\}_{h\in(0,\bar h)}$ is uniformly bounded in $L^2(\Sigma)$, so that, up to a subsequence, we may assume that $\tilde{f}_h \wto \tilde f \in L^2(\Sigma)$. In particular, if $f_h\in\mathcal H_h$ $\cE$-converges to $f\in\mathcal H_0$, then $\tilde{f} := f \circ \Phi_h \in L^2(\Sigma)$ is the pullback of $f$ via $\Phi_h$.

From \eqref{hess_local} we deduce that
\begin{align} \label{eq: form in sigma1}
&(D^2u_h:D^2\varphi)\circ\Phi_h \\
&=\frac{\partial^2_{ss}(u_h\circ\Phi_h)\partial^2_{ss}(\varphi\circ\Phi_h)}{(1-ht\kappa(s))^4}+ \frac{2\partial^2_{st}(u_h\circ\Phi_h)\partial^2_{st}(\varphi\circ\Phi_h)}{h^2(1-ht\kappa(s))^2}+ \frac{1}{h^4}\partial^2_{tt}(u_h\circ\Phi_h)\partial^2_{tt}(\varphi\circ\Phi_h)\label{123}\\
&+\frac{ht\kappa'(s)}{(1-ht\kappa(s))^5}(\partial_s(u_h\circ\Phi_h)\partial^2_{ss}(\varphi\circ\Phi_h)+\partial^2_{ss}(u_h\circ\Phi_h)\partial_s(\varphi\circ\Phi_h))\label{4}\\
&-\frac{\kappa(s)}{h(1-ht\kappa(s))^3}(\partial^2_{ss}(u_h\circ\Phi_h)\partial_t(\varphi\circ\Phi_h)+\partial_t(u_h\circ\Phi_h)\partial^2_{ss}(\varphi\circ\Phi_h)\label{5})\\
&+\frac{2\kappa(s)}{h(1-ht\kappa(s))^3}(\partial_s(u_h\circ\Phi_h)\partial^2_{st}(\varphi\circ\Phi_h)+\partial^2_{st}(u_h\circ\Phi_h)\partial_s(\varphi\circ\Phi_h))\label{6}\\
&-\frac{t\kappa(s)\kappa'(s)}{(1-ht\kappa(s))^4}(\partial_s(u_h\circ\Phi_h)\partial_t(\varphi\circ\Phi_h)+\partial_t(u_h\circ\Phi_h)\partial_s(\varphi\circ\Phi_h))\label{7}\\
&+\frac{(2\kappa(s)^2(1-ht\kappa(s))^2+h^2t^2\kappa'(s)^2)}{(1-ht\kappa(s))^6}\partial_s(u_h\circ\Phi_h)\partial_s(\varphi\circ\Phi_h)\label{8}\\
&+\frac{\kappa(s)^2}{h^2(1-ht\kappa(s))^2}\partial_t(u_h\circ\Phi_h)\partial_t(\varphi\circ\Phi_h)\label{9}.
\end{align}

\textbf{Step 1 (coercivity estimate):} let $\tilde{u}_h = u_h \circ \Phi_h$. We will prove that there exists a constant $C > 0$ such that for all $h \in (0, \bar{h}/2)$
\begin{multline}\label{coerciv_est}
\left(\frac{\norma{\partial^2_{tt} \tilde{u}_h}^2_{L^2(\Sigma)}}{h^4} +  \frac{\norma{\partial^2_{st} \tilde{u}_h}^2_{L^2(\Sigma)}}{h^2} + \frac{\norma{\kappa \partial_{t} \tilde{u}_h}_{L^2(\Sigma)}^2}{h^2} \right) \\
+ \left(\norma{\partial^2_{ss} \tilde{u}_h}^2_{L^2(\Sigma)} + \norma{\kappa \partial_{s} \tilde{u}_h}^2_{L^2(\Sigma)}\right) + \norma{\tilde u_h}^2_{L^2(\Sigma)} \leq C \norma{\tilde{f}_h}_{L^2(\Sigma)}^2
\end{multline}

To shorten the notation, let us set $\rho(s,t):=1-ht\kappa(s)$. Note that since $h\leq\bar h/2$, we that $0<c_1\leq\rho(s,t)\leq c_2$ for all $(s,t)\in\Sigma$, with $c_1,c_2$ independent on $h$. Choose $\varphi=u_{h}$ in \eqref{eq: rescaled prob}. Note that the first three summands of \eqref{eq: form in sigma1}, namely the three terms in \eqref{123}, equal respectively $\left\|\frac{\partial_{ss}\tilde u_h}{\rho^{3/2}}\right\|_{L^2(\Sigma)}^2$, $2\left\|\frac{\partial_{st}\tilde u_h}{h\rho^{1/2}}\right\|_{L^2(\Sigma)}^2$, and $\left\|\rho^{1/2}\frac{\partial_{tt}\tilde u_h}{h^2}\right\|_{L^2(\Sigma)}^2$. The last two terms of \eqref{eq: form in sigma1}, namely \eqref{8} and \eqref{9}, equal $\left(2\left\|\frac{\kappa\partial_{s}\tilde u_h}{\rho^{3/2}}\right\|_{L^2(\Sigma)}^2+\left\|\frac{ht\kappa'\partial_{s}\tilde u_h}{\rho^{5/2}}\right\|_{L^2(\Sigma)}^2\right)$ and $\left\|\frac{\kappa\partial_t\tilde u_h}{h\rho^{1/2}}\right\|_{L^2(\Sigma)}^2$, respectively. As for \eqref{4}, observe that for any $\delta>0$, and any $\varepsilon>0$ sufficiently small, we have
\begin{equation}\label{intermediate}
\begin{split}
&\int_{\Sigma}\frac{2ht\kappa'(s)}{(1-ht\kappa(s))^5}\partial_s\tilde{u}_h \partial^2_{ss}\tilde{u}_h(1-ht\kappa(s))dtds
\geq -C \int_{\Sigma}|\partial_s\tilde u_h\partial^2_{ss}\tilde u_h|dtds \\
&\geq -\frac{C\delta}{2}\|\partial^2_{ss}\tilde u_h\|_{L^2(\Sigma)}^2-\frac{C}{2\delta}\int_0^1\|\partial_s\tilde u_h\|_{L^2((0,|\partial\Omega|))}^2dt\\
& \geq -\frac{C\delta}{2}\|\partial^2_{ss}\tilde u_h\|_{L^2(\Sigma)}^2-\frac{C \varepsilon}{2\delta}\int_0^1\|\partial^2_{ss}\tilde u_h\|^2_{L^2((0,|\partial\Omega|))}dt - \frac{Cc}{2\delta\varepsilon}\int_0^1\|\tilde u_h\|_{L^2((0,|\partial\Omega|))}^2dt\\
&\geq -\frac{C(\delta^2+\varepsilon)}{2\delta}\|\partial^2_{ss}\tilde u_h\|_{L^2(\Sigma)}^2-\frac{CC'c}{2\delta\varepsilon}\int_{\Sigma}\tilde{u}_h^2 (1-ht\kappa(s))dtds
\end{split}
\end{equation}
where $C=\max_{(s,t)\in\overline{\Sigma},h\in[0,\bar  h/2]}\frac{ ht|\kappa'(s)|}{(1- ht\kappa(s))^4}$, $C'=\frac{1}{\min_{(s,t)\in\overline{\Sigma},h\in[0,\bar h/2]}(1- ht\kappa(s))}$. Here, to pass from the first to the second line we have used the Cauchy-Schwarz inequality, and to pass from the second to the third line we have used the classical interpolation inequality $\|v'\|_{L^2((a,b))}\leq\varepsilon \|v''\|_{L^2((a,b))}+\frac{c}{\varepsilon}\|v\|_{L^2((a,b))}$, valid for all $v\in H^2((a,b))$ and $\eps > 0$ sufficiently small, with $c>0$ depending only on $a,b$ (see e.g., \cite[\S 4.2, Theorem 2 and Corollary 7]{burenkov}).

Similarly, we estimate \eqref{6} and \eqref{7} (possibly re-defining the  constants $C,C'$):
\begin{multline}\label{itermediate_2}
\int_{\Sigma}\frac{4\kappa(s)}{h(1-ht\kappa(s))^3}\partial_s\tilde{u}_h \partial^2_{st}\tilde{u}_h (1-ht\kappa(s))dtds\\
\geq -\frac{C\varepsilon}{2\delta}\|\partial^2_{ss}\tilde u_h\|_{L^2(\Sigma)}^2-\frac{CC'c}{2\delta\varepsilon}\int_{\Sigma}\tilde{u}_h^2(1-ht\kappa(s))dtds-\frac{C\delta}{2}\frac{\|\partial^2_{st}\tilde u_h\|_{L^2(\Sigma)}^2}{h^2}
\end{multline}
and
\begin{multline}\label{itermediate_3}
\int_{\Sigma}\frac{2t\kappa(s)\kappa'(s)}{(1-ht\kappa(s))^4}\partial_s \tilde{u}_h \partial_{t} \tilde{u}_h (1-ht\kappa(s))dtds\\
\geq -\frac{C\varepsilon}{2\delta}\|\partial^2_{ss}\tilde u_h\|_{L^2(\Sigma)}^2-\frac{CC'c}{2\delta\varepsilon}\int_{\Sigma}\tilde{u}_h^2(1-ht\kappa(s))dtds-\frac{C\delta}{2}\frac{\|\kappa\partial_t\tilde u_h\|_{L^2(\Sigma)}^2}{h^2}
\end{multline}
where $\delta,\eps>0$ can be chosen arbitrarily small (and independent on $h$), and $C,C',c$ are positive constants not depending on $h$. Finally, we estimate \eqref{5}, which is the most delicate term:
\begin{equation}\label{itermediate_4}
\begin{split}
&\int_{\Sigma}\frac{2\kappa(s)}{h(1-ht\kappa(s))^3}\partial^2_{ss}(\tilde{u}_h)\partial_{t}(\tilde{u}_h)(1-ht\kappa(s))dtds\\
&\qquad \qquad \geq-2\left|\int_{\Sigma}\frac{\kappa(s)}{(1-ht\kappa(s))^2 }\partial_s\tilde u_h\frac{\partial^2_{st} \tilde u_h}{h}+\partial_s\left(\frac{\kappa(s)}{(1-ht\kappa(s))^2}\right)\partial_s\tilde u_h\frac{\partial_t\tilde u_h}{h}dsdt\right|\\
& \qquad \qquad \geq-\frac{C\varepsilon}{2\delta}\|\partial^2_{ss}\tilde u_h\|_{L^2(\Sigma)}^2-\frac{CC'c}{2\delta\varepsilon}\int_{\Sigma}(\tilde{u}_h)^2(1-ht\kappa(s))dtds-\frac{C\delta}{2}\frac{\|\partial^2_{st}\tilde u_h\|_{L^2(\Sigma)}^2}{h^2}\\
& \qquad \qquad \quad -\frac{C\varepsilon}{2\delta'}\|\partial^2_{ss}\tilde u_h\|_{L^2(\Sigma)}^2-\frac{CC'c}{2\delta'\varepsilon'}\int_{\Sigma}(\tilde{u}_h)^2(1-ht\kappa(s))dtds-\frac{C\delta'}{2}\frac{\|\partial_t\tilde u_h\|_{L^2(\Sigma)}^2}{h^2}
\end{split}
\end{equation}
We have used integration by parts and an elementary inequality to pass from the first to the second line of \eqref{itermediate_4}, and interpolation inequalities as done for \eqref{intermediate}, \eqref{itermediate_2}, and \eqref{itermediate_3} to pass from the second line to the last two lines of \eqref{itermediate_4}. The scope of this procedure is to have arbitrarily small coefficients in front of any term involving derivatives, at the price of having a large coefficient in front of the term which does not involve derivatives. Eventually, we will be able to control this term with the constant $M$ which we are free to choose.

Again, $C,C',c$ are positive constants independent on $h$, and $\varepsilon,\varepsilon',\delta,\delta'$ are positive constants which can be chosen arbitrarily. 
In order to conclude, and to establish \eqref{coerciv_est}, the last term in the last line of \eqref{itermediate_4} needs to be bounded from below by  $\frac{\|\kappa \partial_t\tilde u_h\|_{L^2(\Sigma)}^2}{h^2}$.
This is not trivial since $\kappa$ is allowed to vanish on a subset of $\p\Omega$ of positive measure and therefore the two norms are not equivalent. Nevertheless, $\kappa$ is bounded from below by a strictly positive constant on some open subset of $\partial\Omega$. Indeed, the Gauss-Bonnet Theorem implies that $\int_{\partial\Omega}\kappa(s)ds=2\pi$, therefore there exists $a>0$ and an open subset $J_a$ of $\p\Omega$ such that $|\kappa(s)|\geq a$ for all $s\in J_a$.\\[0.1cm]

\textbf{Claim:} there exists $C_a > 0$ such that
\begin{equation}\label{pw3}
\frac{\|\partial_t\tilde u_h\|_{L^2(\Sigma)}^2}{h^2} \leq 4C_a^2 \frac{\|\partial^2_{tt}\tilde u_h\|_{L^2(\Sigma)}^2}{h^2}+4C_a^2\frac{\|\partial^2_{st}\tilde u_h\|_{L^2(\Sigma)}^2}{h^2}+\frac{2}{a^2}\frac{\|\kappa\partial_t\tilde u_h\|_{L^2(\Sigma)}^2}{h^2}.
\end{equation}
\emph{Proof of the Claim.} Let us set $\Sigma_a:=J_a\times(0,1)\subset\Sigma$. Then there exists a constant $C_a>0$ such that $\left\|g-\frac{1}{|\Sigma_a|}\int_{\Sigma_a}g\right\|_{L^2(\Sigma)}\leq C_a\|\nabla g\|_{L^2(\Sigma)}$, for all $g\in H^1(\Sigma)$. This is a general version of the Poincaré-Wirtinger inequality. Thus,
\begin{equation}\label{pw}
\|g\|_{L^2(\Sigma)}=\left\|g-\frac{1}{|\Sigma_a|}\int_{\Sigma_a}g+\frac{1}{|\Sigma_a|}\int_{\Sigma_a}g\right\|_{L^2(\Sigma)}\leq C_a\|\nabla g\|_{L^2(\Sigma)}+\|g\|_{L^2(\Sigma_a)}.
\end{equation}
Inserting $g=\frac{\partial_t \tilde u_h}{h}$ in \eqref{pw} we deduce that
\begin{equation*}
\frac{1}{a}\left\|\frac{\kappa(s)\partial_t\tilde u_h}{h}\right\|_{L^2(\Sigma)}\geq \left\|\frac{\partial_t\tilde u_h}{h}\right\|_{L^2(\Sigma_a)} \geq\left\|\frac{\partial_t\tilde u_h}{h}\right\|_{L^2(\Sigma)}- C_a\left\|\frac{\nabla (\partial_t \tilde u_h)}{h}\right\|_{L^2(\Sigma)},
\end{equation*}
hence \eqref{pw3} holds.\\[0.1cm]

Using \eqref{pw3} on the right-hand side of \eqref{itermediate_4} gives the desired inequality. Note that, choosing suitable $\eps,\eps'\delta,\delta'>0$ in \eqref{intermediate}, \eqref{itermediate_2}, \eqref{itermediate_3}, \eqref{itermediate_4}, and possibly replacing $M$ by a larger (but fixed) constant, we deduce that there exist constants $c_0 > 0$, $c_1 > 1$ independent of $h$ such that
\begin{multline}\label{last_ineq_Q}
\int_{\Sigma} \left(|D^2 \tilde{u}_h|^2  + M \tilde{u}_h^2\right) \rho\,dtds \geq c_0 \left(\frac{\norma{\partial^2_{tt} \tilde{u}_h}^2_{L^2(\Sigma)}}{h^4} +  \frac{\norma{\partial^2_{st} \tilde{u}_h}^2_{L^2(\Sigma)}}{h^2} + \frac{\norma{\kappa \partial_{t} \tilde{u}_h}_{L^2(\Sigma)}^2}{h^2} \right) \\
+ c_0 \left(\norma{\partial^2_{ss} \tilde{u}_h}^2_{L^2(\Sigma)} + \norma{\kappa \partial_{s} \tilde{u}_h}^2_{L^2(\Sigma)}\right) + c_1 \norma{\tilde u_h}^2_{L^2(\Sigma)}
\end{multline}
for all $h\in[0,\bar h/2]$, and since the left-hand side is uniformly bounded in $h$, by  \eqref{eq: rescaled prob}, \eqref{last_ineq_Q} and a standard Cauchy-type estimate, we finally deduce that \eqref{coerciv_est} holds. \\[0.1cm]

\textbf{Step 2 (passage to the limit):} we prove now that there exists $u \in H^2(\Sigma)$, $w \in H^1(\Sigma)$, such that $\tilde{u}_h \to u$ in $L^2(\Sigma)$, $\frac{\partial_{t} \tilde{u}_h}{h} \to  w$ in $L^2(\Sigma)$. Moreover, $u, w$ are both constant in the variable $t$, and $(u,w)$ solves
\begin{equation} \label{eq: limiting problem}
\begin{cases}
u''''-2(\kappa^2 u')'-(\kappa w)''-2(\kappa w')'+Mu=\mathcal M\tilde f, & {\rm in\ }(0,|\partial\Omega|),\\
-2w''+\kappa^2 w-\kappa u''-2(\kappa u')'=0, & {\rm in\ } (0,|\partial\Omega|),\\
u^{(k)}(0)=u^{(k)}(|\partial\Omega|), & k=0,1,2,3,\\
w^{(k)}(0)=w^{(k)}(|\partial\Omega|), & k=0,1.
\end{cases}\vspace{0.2cm},
\end{equation}
where $\mathcal M$ is the averaging operator defined as
\[
\cM f(s) = \int_{0}^1 \tilde f(s, t) dt, \quad \textup{ a.a. $s \in (0, |\p \Omega|)$}.
\]
Step 1 implies that the sequences
\begin{equation} \label{eq: bounded seq}
\left\{\frac{\partial^2_{tt} \tilde{u}_h}{h^2}\right\}_{h\in(0,\bar h)},  \left\{\frac{\partial^2_{st} \tilde{u}_h}{h}\right\}_{h\in(0,\bar h)}, \left\{\frac{\kappa \partial_{t} \tilde{u}_h}{h}\right\}_{h\in(0,\bar h)}, \left\{ \partial^2_{ss} \tilde{u}_h\right\}_{h\in(0,\bar h)}, \left\{\kappa \partial_{s} \tilde{u}_h\right\}_{h\in(0,\bar h)}
\end{equation}
are uniformly bounded in $L^2(\Sigma)$ for all $h\in(0,\bar h)$. In particular, $\{\tilde{u}_h\}_{h \in (0,\bar h)}$ is a bounded sequence in $H^2(\Sigma)$. By the compact embedding of $H^2(\Sigma)$ in $L^2(\Sigma)$ we  deduce that there exists a function $u \in H^2(\Sigma)$ such that, up to a subsequence, $\tilde{u}_h \wto u$ in $H^2(\Sigma)$, strongly in $H^1(\Sigma)$. Note also that there exists a function $v\in L^2(\Sigma)$ such that, up to a subsequence, \begin{equation}\label{eq: weak limits 2}
 \frac{\partial^2_{tt} \tilde{u}_h}{h^2} \wto v
\end{equation}
in $L^2(\Sigma)$ as $h \to 0^+$.

Moreover, the sequence $\left\{\frac{ \partial_{t} \tilde{u}_h}{h}\right\}_{h\in(0,\bar h)}$ is uniformly bounded in $L^2(\Sigma)$. This follows from \eqref{pw3} and \eqref{eq: bounded seq}. Then, up to a subsequence, there exists a function $w \in H^1(\Sigma)$ such that
\begin{equation}
\label{eq: weak limits}
\frac{\partial_{t} \tilde{u}_h}{h} \wto  w
\end{equation}
in $H^1(\Sigma)$ as $h \to 0^+$, and, from the compact embedding of $H^1(\Sigma)$ in $L^2(\Sigma)$, $\frac{\partial_{t} \tilde{u}_h}{h} \to  w$ in $L^2(\Sigma)$.  In particular, $\partial_t \tilde{u}_h \to 0$ in $L^2(\Sigma)$, hence $\partial_t u = 0$ a.e. in $\Sigma$, so $u$ is constant in $t$.

We further deduce that the limiting function $w$ is constant in $t$, due to the fact that $\p^2_{tt} \tilde{u}_h/h \to 0$ in $L^2(\Sigma)$. We are now in position to pass to the limit in equation \eqref{eq: rescaled prob}. This will be done in three steps.\\[0.1cm]
\textbf{Step 2a:} we first choose $\varphi = h^2 \varsigma$ for some $\varsigma \in H^2(\Sigma)$. Then all the summands in \eqref{eq: rescaled prob}, which are listed in \eqref{eq: form in sigma1}, vanish as $h \to 0^+$, with the possible exception of
\[\frac{1}{h^4}\partial^2_{tt}(u_h\circ\Phi_h)\partial^2_{tt}(h^2 \varsigma).\]
From \eqref{eq: weak limits 2} and from equation \eqref{eq: rescaled prob} we then deduce that
\[
\int_{\Sigma} (D^2 u_h:D^2(h^2\varsigma))\circ\Phi_h(s,t)(1-th\kappa(s))\,dtds \to \int_{\Sigma} v\, \partial^2_{tt}\varsigma \,dsdt = 0
\]
as $h\rightarrow 0^+$. Since $\varsigma$ is an arbitrary function in $H^2(\Sigma)$, we conclude that $v = 0$. \\[0.1cm]
\textbf{Step 2b:} we now choose $\varphi(s,t) = h t \theta(s)$, for $(s,t) \in \Sigma$, where $\theta \in H^2_{p}((0, |\p \Omega|))$. Using $\varphi$ as test function in \eqref{eq: rescaled prob} we deduce that
\begin{multline*}
\int_{\Sigma} \left(\frac{2\partial^2_{st}(\tilde{u}_h)\partial^2_{st}\varphi}{h^2(1-ht\kappa(s))^2}
-\frac{\kappa(s)}{h(1-ht\kappa(s))^3}\partial^2_{ss}(\tilde{u}_h)\partial_t\varphi \right.\\
\left.+\frac{2\kappa(s)}{h(1-ht\kappa(s))^3}\partial_s(\tilde{u}_h)\partial^2_{st}\varphi -\frac{\kappa(s)^2}{h^2(1-ht\kappa(s))^2}\partial_t(\tilde{u}_h)\partial_t\varphi\right)(1-ht\kappa(s))\,dsdt = o(1)
\end{multline*}
as $h\rightarrow 0^+$. Recalling \eqref{eq: weak limits} and the specific choice of $\varphi$ we can now pass to the limit in the previous equation to deduce that
\begin{equation}\label{weak_w}
\int^{|\p \Omega|}_{0} 2 w'(s)\theta'(s) - \kappa(s) u''(s)\theta(s) + 2 \kappa(s) u'(s) \theta'(s) + \kappa^2 w(s) \theta(s)\,ds = 0
\end{equation}
for all $\theta \in H^2_{p}((0, |\p \Omega|))$, and, by approximation, for all $\theta\in H^1_p((0,\p \Omega))$.

Since the coefficient of the leading term in \eqref{weak_w} is constant, $u\in H^2_p((0,|\p\Omega|))$, and $\kappa$ is smooth, we deduce that $w\in H^2_p(\Sigma)$ and solves
\begin{equation}
\label{eq: limit eq 1}
-2 w'' + \kappa^2 w- \kappa u'' - 2 (\kappa u')'  = 0 \ \ \ {\rm on\ }(0,|\p\Omega|),
\end{equation}
where the equality is understood in the $L^2((0,|\p\Omega|))$ sense.\\[0.1cm]
\textbf{Step 2c:} we finally choose $\varphi(s,t) = \psi(s)$, for $s \in (0, |\p \Omega|)$, $\psi \in H^2_{p}((0, |\p \Omega|))$. Using $\varphi$ as test function in \eqref{eq: rescaled prob} we deduce that
\begin{multline*}
\int_{\Sigma}\left( \frac{\partial^2_{ss}\tilde{u}_h\partial^2_{ss}\varphi}{(1-ht\kappa(s))^4}
+\frac{ht\kappa'(s)}{(1-ht\kappa(s))^5}(\partial_s\tilde{u}_h\partial^2_{ss}\varphi+\partial^2_{ss}\tilde{u}_h\partial_s\varphi)\right.\\
-\frac{\kappa(s)}{h(1-ht\kappa(s))^3}(\partial_t\tilde{u}_h\partial^2_{ss}\varphi)
+\frac{2\kappa(s)}{h(1-ht\kappa(s))^3}(\partial^2_{st}\tilde{u}_h\partial_s\varphi)
-\frac{t\kappa(s)\kappa'(s)}{(1-ht\kappa(s))^4}(\partial_t\tilde{u}_h\partial_s\varphi) \\
\left.+\frac{(2\kappa(s)^2(1-ht\kappa(s))^2+h^2t^2\kappa'(s)^2)}{(1-ht\kappa(s))^6}\partial_s\tilde{u}_h\partial_s\varphi + M \tilde{u}_h \varphi\right)(1-th\kappa(s))\,dsdt\\
 = \int_{\Sigma} \tilde{f}_h \varphi (1-th\kappa(s))\,dsdt
\end{multline*}
and taking the limit as $h \to 0^+$ we deduce that
\begin{equation}\label{eq:step3}
\int_{\Sigma} \partial^2_{ss}u \partial^2_{ss}\psi  -\kappa(s) w \partial^2_{ss}\psi + 2\kappa(s)(\partial_{s}w \partial_s\psi) + 2\kappa(s)^2 \partial_s u \partial_s\psi + M u \psi\,  dsdt= \int_{\Sigma} \tilde f \psi\, dsdt
\end{equation}
Note that all the functions appearing in \eqref{eq:step3} are constant in $t$, with the possible exception of $\tilde f$.
As in Step 2, we deduce that $u\in H^4_p(\Sigma)$ and solves
\begin{equation} \label{eq: limit eq 2}
u''''-2(\kappa^2 u')'-(\kappa w)''-2(\kappa w')'+M u=\mathcal M\tilde f\ \ \ {\rm on\ }(0,|\p\Omega|).
\end{equation}
A standard bootstrap argument allows to conclude that $u,w$ are smooth. Altogether, we have found that the solution $\tilde{u}_h$ of \eqref{eq: rescaled prob} converges as $h \to 0^+$ to the solution $u$ of the system \eqref{eq: limiting problem}.
We can rewrite \eqref{eq: limiting problem} as a single equation by noting that the operator $\Delta_{\kappa}$ has a bounded inverse, so the second equation in \eqref{eq: limiting problem} yields
\[
w =  \Delta_{\kappa}^{-1}\left(\frac{\kappa}{2} u'' + (\kappa u')'\right),
\]
and upon substitution in the first equation in \eqref{eq: limiting problem} we recover \eqref{eq: limit problem}.\\[0.1cm]

{\bf Step 3 (proof of the compact convergence).}
From Steps 1-2 we see that if $f_h$ $\cE$-converges to $f \in \cH_0$ then
\[
\norma{B_h f_h - \cE_h B_0 f}_{\cH_h} = \norma{\tilde{u}_h - \cE_h u}_{L^2(\Sigma)} \to 0\,,\ \ \ {\rm as\ }h\rightarrow 0^+,
\]
and similarly, if $\{f_h\}_{h\in(0,\bar h)}$ is uniformly bounded in the sequence of Hilbert spaces $\cH_h$, with $\tilde{f}_h \wto \tilde f$ in $L^2(\Sigma)$, then from the considerations above, $\tilde{u}_h \wto u$ in $H^2(\Sigma)$ and $\tilde{u}_h \to u$ strongly in $L^2(\Sigma)$ so
\[
\norma{B_h f_h - \cE_h B_0 f}_{\cH_h} = \norma{\tilde{u}_h - \cE_h u}_{L^2(\Sigma)} \to 0\,,\ \ \ {\rm as\ }h\rightarrow 0^+,
\]
concluding the proof.
\end{proof}

\begin{remark}
Since we know that $\mu_1(h) = M$, we deduce a posteriori that the operator $\tilde A_0=A_0 - M$ is non-negative in $L^2((0, |\p \Omega|))$.
\end{remark}

\section{Final remarks}\label{remarks}

\subsection{Tubular neighbourhoods with variable size}
It is possible to consider, instead of $\omega_h$, a tubular neighbourhood of $\partial\Omega$ of variable size, namely
\begin{equation}\label{omega_h_g}
\omega_{h,g}:=\{x\in\omega_h: 0<{\rm dist}(x,\partial\Omega)<hg(s(x))\},
\end{equation}
for all $h\in (0,\bar h)$, where $s(x)$ is the nearest point to $x$ on $\partial\Omega$, and $g:\partial\Omega\rightarrow\mathbb R$ is a smooth function such that $0<g(s)<1$ for all $s\in\partial\Omega$. The computations can be carried out exactly as in the previous section. In particular, it follows that the limiting problem of  \eqref{classic_N} with $\omega_h$ replaced by $\omega_{h,g}$ reads
\begin{equation} \label{eq: limiting problem:g}
\begin{cases}
(gu'')'' - (\kappa g w)'' - 2 (\kappa g w')' - 2(k^2 u')'=\eta g u, \quad &\textup{in $(0, |\p \Omega|)$} \\
-2 (gw')' + \kappa^2 g w- \kappa g u'' - 2 (\kappa g u')'  = 0, \quad &\textup{in $(0, |\p \Omega|)$},\\
u^{(k)}(0)=u^{(k)}(|\partial\Omega|), & k=0,1,2,3,\\
w^{(k)}(0)=w^{(k)}(|\partial\Omega|), & k=0,1,
\end{cases}
\end{equation}
in the unknowns $u(s),w(s)$ and $\eta$ (the eigenvalue).
We refer e.g., to \cite[\S 4]{AFL17} for more details in the case of a thin set of the form $\{(x,y)\in\mathbb R^2:0<x<1, 0<y<hg(x)\}$.

\subsection{The unit circle}\label{disk}
Let us consider the case when $\Omega$ is the unit disk in $\mathbb R^2$. Then, for $h\in(0,1)$, $\omega_h=\{x\in\mathbb R^2:1-h<|x|<1\}$ is an annulus of width $h$. As customary, we look for solutions to problem \eqref{classic_N} of the form
\begin{equation}\label{eigen_annulus}
u_h^{\ell}(r,\theta)=v_h(r)(A_{\ell}\cos(\ell\theta)+B_{\ell}\sin(\ell\theta))
\end{equation}
Here we are using polar coordinates $(r,\theta)$ in $\mathbb R^2$. Plugging \eqref{eigen_annulus} in \eqref{classic_N} we obtain that the radial part $v_h$ satisfies the following ODE
\begin{equation}\label{classic_disk}
\begin{cases}
v_h''''+\frac{2v_h'''}{r}-\frac{(1+2\ell^2)v_h''}{r^2}+\frac{(1+2\ell^2)v_h'}{r^3}+\frac{\ell^2(\ell^2-4)v_h}{r^4}=\mu(h) v_h\,, & r\in(1-h,1),\\
v_h''(1-h)=v_h''(1)=0\,,\\
v_h'''(1)+v_h''(1)-(1+2\ell^2)v_h'(1)+3\ell^2 v_h(1)=0\,,\\
v_h'''(1-h)+\frac{v_h''(1-h)}{1-h}-\frac{(1+2\ell^2)v_h'(1-h)}{(1-h)^2}+\frac{3\ell^2 v_h(1-h)}{(1-h)^3}=0.
\end{cases}
\end{equation}
For readers interested in more details on how to obtain \eqref{classic_disk} we refer e.g., to \cite[\S 6]{chasmancircular}.

\begin{figure}[h]
\centering
\includegraphics[width=0.5\textwidth]{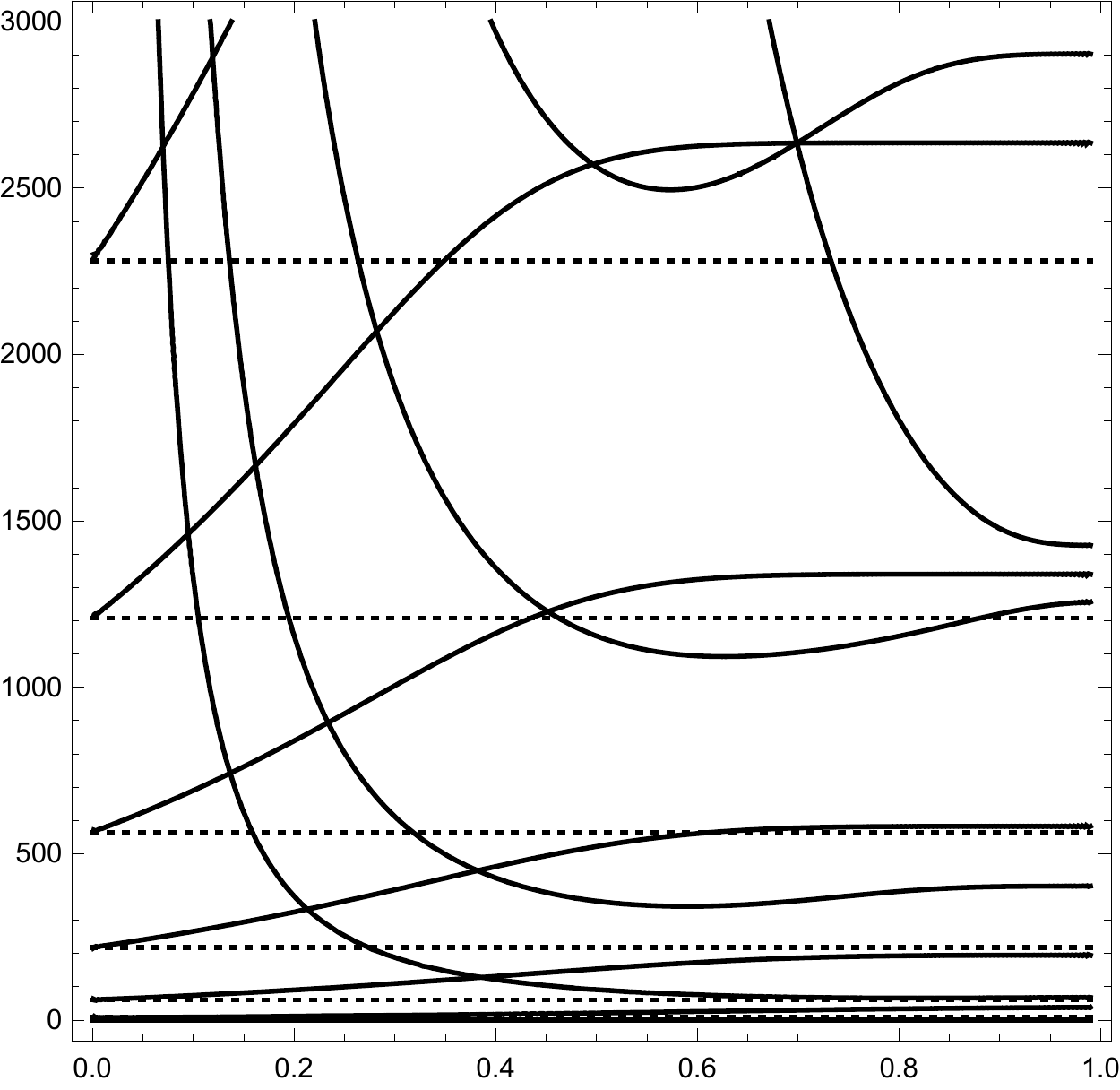}
\caption{Each analytic curve $\mu=\mu(h)$ describes the zero level set of ${\rm det}\mathcal B_{\ell}(h,\mu)$, for $\ell=0,...,7$. The horizontal dotted lines correspond to $\frac{2\ell^2(\ell^2-1)^2}{1+2\ell^2}$, $\ell=0,...,7$. The plot suggests that $\mu_j(h)$ are locally monotone near $ h=0$, see also {\rm\protect\cite{LP_ed}} for a related monotonicity result.}
\end{figure}

It is customary to verify that any solution of the differential equation in \eqref{classic_disk} is of the form
\begin{equation}\label{form_bessel}
v_h^{\ell}(r)=a_{\ell} J_{\ell}(\mu(h)^{1/4}r)+b_{\ell} I_{\ell}(\mu(h)^{1/4}r)+c_{\ell} Y_{\ell}(\mu(h)^{1/4}r)+d_{\ell} K_{\ell}(\mu(h)^{1/4}r),
\end{equation}
where  $J_{\ell},Y_{\ell}$ denote the Bessel function of first and second order of degree $\ell$, respectively, and $I_{\ell},K_{\ell}$ denote the modified Bessel function of first and second order of degree $\ell$, respectively. We refer e.g., to \cite[Prop. 1]{chasmancircular} and to \cite{LP_ed} for the justification of \eqref{form_bessel}.

Imposing the four boundary conditions we obtain a homogeneous system of four equations in four unknowns $a_{\ell},b_{\ell},c_{\ell},d_{\ell}$, which admits a non-zero solution if and only if the determinant of the associated matrix is zero. Namely, a number $\mu$ is an eigenvalue of \eqref{classic_disk} corresponding to an index $\ell\in\mathbb N$ and to $h\in(0,1)$ if and only if
$$ {\rm det}\mathcal B_{\ell}(h,\mu)=0,$$
where
\scriptsize
\begin{multline*}
\mathcal B_{\ell}(h,\mu):=\\
\begin{pmatrix}
J_{\ell}''(\mu^{1/4}) & I_{\ell}''(\mu^{1/4}) & Y_{\ell}''(\mu^{1/4}) & K_{\ell}''(\mu^{1/4})\\
\ &\ &\ &\\
J_{\ell}''(\mu^{1/4}(1-h)) & I_{\ell}''(\mu^{1/4}(1-h)) & Y_{\ell}''(\mu^{1/4}(1-h)) & J_{\ell}''(\mu^{1/4}(1-h))\\
\ &\ &\ &\\
J_{\ell}'''(\mu^{1/4}) & I_{\ell}'''(\mu^{1/4}) & Y_{\ell}'''(\mu^{1/4}) & K_{\ell}'''(\mu^{1/4})\\
-(1+2\ell^2)J_{\ell}'(\mu^{1/4}) & -(1+2\ell^2)I_{\ell}'(\mu^{1/4}) & -(1+2\ell^2)Y_{\ell}'(\mu^{1/4}) & -(1+2\ell^2)K_{\ell}'(\mu^{1/4})\\
+3\ell^2J_{\ell}(\mu^{1/4}) & +3\ell^2I_{\ell}(\mu^{1/4}) & +3\ell^2Y_{\ell}(\mu^{1/4}) & +3\ell^2K_{\ell}(\mu^{1/4})\\
\ &\ &\ &\\
J_{\ell}'''(\mu^{1/4}(1-h)) & I_{\ell}'''(\mu^{1/4}(1-h)) & Y_{\ell}'''(\mu^{1/4}(1-h)) & K_{\ell}'''(\mu^{1/4}(1-h))\\
-\frac{(1+2\ell^2)J_{\ell}'(\mu^{1/4}(1-h))}{(1-h)^2} & -\frac{(1+2\ell^2)I_{\ell}'(\mu^{1/4}(1-h))}{(1-h)^2} & -\frac{(1+2\ell^2)Y_{\ell}'(\mu^{1/4}(1-h))}{(1-h)^2} & -\frac{(1+2\ell^2)K_{\ell}'(\mu^{1/4}(1-h))}{(1-h)^2}\\
+\frac{3\ell^2J_{\ell}(\mu^{1/4}(1-h))}{(1-h)^3} & +\frac{3\ell^2I_{\ell}(\mu^{1/4}(1-h))}{(1-h)^3} & +\frac{3\ell^2Y_{\ell}(\mu^{1/4}(1-h))}{(1-h)^3} & +\frac{3\ell^2K_{\ell}(\mu^{1/4}(1-h))}{(1-h)^3}\\
\end{pmatrix}
\end{multline*}
\normalsize
Expanding the determinant in Taylor series with respect to $h$ near $h=0$, and using recurrence relations for Bessel functions and cross-products formulae, we obtain
$$
{\rm det}\mathcal B_{\ell}(h,\mu)=\frac{8\mu\left(\mu(1+2\ell^2)-2\ell^2(\ell^2-1)^2\right)}{\pi}h^2+O(h^3)\,,\ \ \ h\rightarrow 0^+.
$$
This implies that the limiting eigenvalues are of the form $\frac{2\ell^2(\ell^2-1)^2}{1+2\ell^2}$, see Fig.2. The computations, which we omit, are very long and technical. The reader interested in the details may refer to \cite{LP_ed} where analogous computations were performed in the case of a singularly perturbed eigenvalue problem for the Neumann Laplacian with density on a thin annulus.

On the other hand, choosing $\kappa(s)\equiv 1$ in \eqref{classic_system}, which corresponds to the case of the unit circle, it is standard to prove that all solutions of \eqref{classic_system} are given by $u(s)=A\cos(\ell\theta)+B\sin(\ell\theta)$, $w(s)=-\frac{3\ell^2}{1+2\ell^2}u(s)$, and $\eta=\frac{2\ell^2(\ell^2-1)^2}{1+2\ell^2}$, for $\ell\in\mathbb N$ and arbitrary constants $A,B\in\mathbb R$.

\subsection{On the restriction of the biharmonic operator on functions depending only on the tangential curvilinear coordinate}
It is well-known that the equality $\Delta_{\partial\Omega} u=\Delta U_{|_{\partial\Omega}}$ holds for all functions $u$ defined on $\partial\Omega$, where $U$ is defined on $\omega_h$ by $(U\circ\Phi_h)(s,t)=u(s)$ for all $(s,t)\in\Sigma=\Phi_h^{-1}(\omega_h)$. Here $\Delta_{\partial\Omega}$ denotes the second derivative with respect to the arc-length parameter $s$, since we are working in two space dimensions. The same identification is possible in any dimension $n\geq 2$. Namely, the Laplace-Beltrami operator acting on a function $u$  defined on a closed hypersurface $\partial\Omega$ in $\mathbb R^n$ bounding a smooth domain is the restriction to $\partial\Omega$ of the Laplacian acting on the function $U$ defined in a tubular neighbourhood of $\partial\Omega$ by extending $u$ constantly in the normal direction.

This identification is no longer true in the case of the biharmonic operator. In fact, by direct inspection one sees that $\Delta^2 U_{|_{\partial\Omega}}$ turns out to have an explicit representation, which for $n=2$ reads
\begin{equation}\label{limit_2}
\Delta^2 U_{|_{\partial\Omega}}=u''''(s)+4\kappa^2(s)u''(s)+5\kappa(s)\kappa'(s)u'(s),
\end{equation}
$s\in(0,|\partial\Omega|)$. Note that the corresponding differential operator coincides neither with $\Delta^2_{\p\Omega}$, nor with the operator associated with problem \eqref{classic_system}. This discrepancy is due to the different behaviour of the Laplace operator and the biharmonic operator on thin domains, as we have seen in the proof of Theorem \ref{main_2}; $\Delta^2$ change drastically in the limit if we neglect the contributions coming from normal derivatives, differently to what happens in the case of the Laplacian. This fact can also be deduced by the following remark: functions defined in $\omega_h$ which depend only on the tangential curvilinear coordinate $s$ do not satisfy in general the second boundary condition in \eqref{classic_N}.  In fact, if $f:\Sigma\rightarrow\mathbb R$ depends only on $s$, the second boundary condition for $f\circ\Phi_h^{-1}$ in \eqref{classic_N} reads, in coordinates $(s,t)$,
\begin{equation} \label{eq: second bc}
-\left(\frac{\kappa(s)}{(1-h t\kappa(s))^2}f'(s)\right)'_{|_{t=0,1}}=0,
\end{equation}
while the first boundary condition is trivially satisfied. Note that if $\Omega$ is strictly convex, then \eqref{eq: second bc} implies that $f' = 0$, hence $f$ is constant. On the other hand, in the case of the Neumann Laplacian on $\omega_h$, the boundary condition is trivially satisfied by any $f$ depending only on $s$.

\section*{Acknowledgements}\label{ackref}
The authors would like to thank the two referees for their remarks, which have substantially improved a previous version of this article. 
The first author acknowledges the support of the `Engineering and Physical Sciences Research Council' (EPSRC) through the grant EP/T000902/1, \textit{`A new paradigm for spectral localisation of operator pencils and analytic operator-valued functions'}.
The second author is member of the Gruppo Nazionale per le Strutture Algebriche, Geometriche e le loro Applicazioni (GNSAGA) of the I\-sti\-tuto Naziona\-le di Alta Matematica (INdAM).

\def\cprime{$'$} \def\cprime{$'$} \def\cprime{$'$} \def\cprime{$'$}
  \def\cprime{$'$}

\end{document}